\documentclass[11pt]{article}

\usepackage{amsmath}
\usepackage{amsfonts}
\usepackage{color}
\usepackage{float}   
\usepackage{amssymb}
\usepackage{bm}    
\usepackage{graphicx}   
\usepackage{enumerate}
\usepackage{amsthm,amscd}
\usepackage[all]{xy}
\usepackage{multirow,makecell}
\usepackage{pifont}        
 \usepackage{arydshln}    
 \usepackage{booktabs}  

 \usepackage{tikz,tcolorbox}

\makeatletter
\def\hlinew#1{%
  \noalign{\ifnum0=`}\fi\hrule \@height #1 \futurelet
   \reserved@a\@xhline}
\makeatother       

 \allowdisplaybreaks

\usepackage{hyperref}

\newtheorem{theorem}{Theorem}[section]

\newtheorem{conjecture}[theorem]{Conjecture}

\newtheorem{problem}[theorem]{Problem}
\newtheorem{lemma}[theorem]{Lemma}

\begin{document}

\title{A spectral Erd\H{o}s-Rademacher theorem\thanks{
This paper was published on 
Advances in Applied Mathematics 158 (2024), No. 102720. 
This is the final version; see \url{https://doi.org/10.1016/j.aam.2024.102720}. 
The research was supported by the NSFC grant 11931002 and 12371362.
E-mail addresses: \url{ytli0921@hnu.edu.cn} (Y. Li),
\url{lulugdmath@163.com} (L. Lu),
\url{ypeng1@hnu.edu.cn} (Y. Peng, corresponding author).}  }

\author{Yongtao Li$^{\dag}$, Lu Lu$^{\dag}$, Yuejian Peng$^{\ddag}$ \\[2ex]
{\small $^{\dag}$School of Mathematics and Statistics, Central South University} \\
{\small Changsha, Hunan, 410083, P.R. China}\\
{\small $^{\ddag}$School of Mathematics, Hunan University} \\
{\small Changsha, Hunan, 410082, P.R. China }
}

\maketitle

\vspace{-0.5cm}

\begin{abstract}
A classical result of Erd\H{o}s and Rademacher (1955)  indicates a supersaturation phenomenon. It says that if $G$ is a graph on $n$ vertices with at least $\lfloor {n^2}/{4} \rfloor +1$ edges,
then $G$ contains at least $\lfloor {n}/{2}\rfloor$ triangles. We prove a spectral version of Erd\H{o}s--Rademacher's theorem. 
Moreover, Mubayi [Adv. Math. 225 (2010)] extends the result of  Erd\H{o}s and Rademacher from a triangle  to any  color-critical graph. It is interesting to  
study the extension of Mubayi from a spectral perspective. However,  it is not apparent to measure the increment on the spectral radius of a graph comparing to the traditional edge version (Mubayi's result). In this paper, we provide a way to measure the increment on the spectral radius of a graph and  propose a spectral version on the counting problems for color-critical graphs.
\end{abstract}



\section{Introduction}

We shall use the following standard notation; see, e.g., the monograph \cite{BM2008}. Let $G$ be a simple
 graph with vertex set $V(G)=\{v_1, \ldots, v_n\}$ and edge set $E(G)=\{e_1, \ldots, e_m\}$.
 We write $n$ and $m$ for the number of vertices and edges,
 respectively, and call the order and size of $G$.
Let $N(v)$ or $N_G(v)$ be the set of neighbors of $v$,
and $d(v)$ or $d_G(v)$ be the degree of a vertex $v$ in $G$.
We write $t(G)$ for the number of triangles in $G$.
For a subset $A\subseteq V(G)$, we write $e(A)$ for the
number of edges with two endpoints in $A$, and $G[A]$
for the subgraph of $G$ induced by $A$.
For a vertex  $w\in V(G)$,
we denote by $d_A(w)$
the number of neighbors of $w$ in $A$.
For disjoint sets $A$ and $B$, we write $e(A,B)$ for the number of edges between $A$ and $B$.
  Let $K_n$ be the complete graph on $n$ vertices, and
  $K_{s,t}$ be the complete bipartite graph with parts of sizes
  $s$ and $t$.
 We write  $C_n$ and $P_n$ for the cycle and
 path on $n$ vertices, respectively.
The  {\it Tur\'{a}n graph} $T_{n,r}$
is an $n$-vertex complete $r$-partite graph with each part of size $n_i$ such that
$|n_i-n_j|\le 1$ for all $1\le i,j\le r$, that is, $n_i$ equals
$\lfloor {n}/{r} \rfloor$  or $\lceil {n}/{r}\rceil$.
Sometimes, we  call $T_{n,r}$ the {\it balanced complete $r$-partite graph}.
In particular, we have $T_{n,2}=K_{\lfloor \frac{n}{2} \rfloor, \lceil \frac{n}{2}\rceil}$.

\subsection{Counting substructures for graphs with given size}

A graph $G$ is called $F$-free if it does not contain
  an isomorphic copy of $F$ as a subgraph.
  Apparently, every bipartite graph is $C_{2k+1}$-free
  for every  $k\ge 1$.
   The {\em Tur\'{a}n number} of a graph $F$ is the maximum number of edges  in an $n$-vertex $F$-free graph, and
  it is usually  denoted by $\mathrm{ex}(n, F)$.
  An $n$-vertex $F$-free graph with $\mathrm{ex}(n, F)$ edges is called an {\em extremal graph} for $F$.
The study of Tur\'{a}n problem goes back to
 the work of Mantel  \cite{Man1907}, which states that
 if $G$ is a triangle-free graph on $n$ vertices, then
$  e(G) \le    \lfloor {n^2}/{4} \rfloor $,
and equality holds if and only if  $G$ is the balanced complete bipartite graph $T_{n,2}$.
Mantel's theorem has many interesting applications and generalizations in the literature; see
 \cite[p. 294]{Bollobas78} for standard proofs.
 Inspired by Mantel's theorem,
 the extremal graph theory was widely studied and achieved a rapid development in the past years; see  \cite{FS13,Sim13} for comprehensive surveys.
 In particular, Rademacher (unpublished, see \cite{Erd1955}) extended Mantel's result
 by showing that any graph
 with  $\lfloor n^2/4\rfloor +1$
 edges has not only one triangle, but also at least $\lfloor {n}/{2}\rfloor$ triangles.

 \begin{theorem}[Erd\H{o}s--Rademacher \cite{Erd1955}, 1955]  \label{thmrad}
If $G$ is a graph on $n$ vertices with
\begin{equation*} \label{eq-ER}
  e(G) \ge \left\lfloor \frac{n^2}{4} \right\rfloor +1,
  \end{equation*}
then $G$ contains at least $\lfloor \frac{n}{2}\rfloor$ triangles.
\end{theorem}

To see the sharpness of Theorem \ref{thmrad},
one can consider the graph obtained
by adding one edge to the larger color class of  $T_{n,2}$.
Moreover, $T_{n,2}$ is the unique graph with $\lfloor n^2/4\rfloor$ edges which contains no triangle.
A special case of Erd\H{o}s \cite[Theorem 1]{Erdos1964} reveals that if $e(G) \ge \lfloor n^2/4\rfloor$,
then
\begin{equation} \label{eq-Erdos-1}
t(G) \ge \left\lfloor \frac{n}{2} \right\rfloor -1,
\end{equation}
unless $G$ is the bipartite Tur\'{a}n graph $T_{n,2}$.

\medskip
A graph is called {\it color-critical} if its chromatic  number can be decreased by removing an edge.
This is a large and important class of graphs for which
the Tur\'{a}n number is well understood.
It was shown by Simonovits \cite{Sim1966} that
for every  color-critical  graph $F$ with $\chi (F)=r+1 \ge 3$,
 there exists an integer $n_0=n_0(F)$
such that if $n\ge n_0$ and $G$ is an $n$-vertex $F$-free graph, then
\begin{equation} \label{eq-Sim}
   e(G) \le e(T_{n,r}).
   \end{equation}
Moreover,
the equality holds if and only if $G$ is the
$r$-partite Tur\'{a}n graph $T_{n,r}$.
In 2010, Mubayi \cite{Mubayi2010} extended Simonovits' result by proving the following  counting result.

\begin{theorem}[Mubayi \cite{Mubayi2010}, 2010] \label{thm-Mubayi}
Let  $F$ be a color-critical graph
with $\chi (F)=r+1 \ge 3$. There exists $\delta =\delta (F)>0$ such that
if $n$ is sufficiently large, $1\le q < \delta n$, and
$G$ is an $n$-vertex graph with
\begin{equation*} \label{eq-Mub}  e(G)\ge e(T_{n,r}) +q,
\end{equation*}
then $G$ contains at least $q\cdot c(n,F) $ copies of $F$,
where $c(n,F)$ is the minimum number of copies of $F$
in the graph obtained from $T_{n,r}$ by adding an edge.
\end{theorem}

\subsection{Counting triangles using the spectral radius}

Let $G$ be a simple graph on vertex set $\{v_1,\ldots ,v_n\}$.
Let $A(G)=[a_{ij}]_{i,j=1}^n$ be the adjacency matrix of $G$
 with the entry $a_{ij}=a_{ji}=1$ if
$v_i$ and $v_j$ are adjacent, and $a_{ij}=a_{ji}=0$ otherwise.
Let $\lambda (G)$ be the spectral radius of $G$,
which is defined as the maximum of
modulus of eigenvalues of  $A(G)$.
The classical extremal graph problems
 usually study the maximum or minimum
number of edges that the extremal graphs can have.
Correspondingly,
we can study the spectral Tur\'{a}n type problem,
which asks for the maximum value $\lambda (G)$ over all
$F$-free graphs $G$ with $n$ vertices.
In particular, Nikiforov  \cite{Niki2007laa2} published the spectral Tur\'{a}n theorem
by showing that
if  $G$ is a $K_{r+1}$-free graph on $n$ vertices, then
$ \lambda (G)\le \lambda (T_{n,r})$,
with equality if and only if $G =T_{n,r}$.
In the past few years, the spectral Tur\'{a}n theorem has stimulated the developments of the spectral extremal graph theory.
Various extensions and generalizations  have been obtained in the literature; see, e.g.,
\cite{Wil1986,Niki2002cpc,Niki2007laa2} for extensions
on $K_{r+1}$-free graphs;
see \cite{BN2007jctb} for relations
between cliques and spectral radius,
\cite{TT2017,LN2021outplanar} for outerplanar and planar graphs,
\cite{LNW2021,LP2022second,LFP2023} for non-bipartite triangle-free graphs,
\cite{LG2021,LSY2022} for non-bipartite graphs without short odd cycles,
\cite{CFTZ20, ZLX2022, 2022LLP} for intersecting triangles,
\cite{CDT2022a} for odd wheels,
\cite{CDT2022b} for even cycles,
\cite{CDT2023} for spectral Erd\H{o}s--S\'{o}s theorem,
\cite{ZLS2021} for $C_5$-free, $C_6$-free and $K_{2,r+1}$-free graphs with given size,
\cite{ZL2022jgt} for books and theta graphs,
\cite{Tait2019} for $K_r$-minor-free graphs,
\cite{ZL2022jctb} for $K_{s,t}$-minor-free graphs,
\cite{WKX2023} for a spectral extremal conjecture, 
\cite{C2006, Lu2012,XLGS2023} for eigenvalues and regularity, 
and \cite{NikifSurvey} for a comprehensive survey.

\medskip  
In 2009, Nikiforov \cite{Niki2009ejc} proved
 a spectral generalization of
the Simonovits theorem.
More precisely,
Nikiforov's result implies that if $F$ is a color-critical graph with $\chi (F)=r+1$ where $r\ge 2$,
then there exists an $n_0=n_0(F)$
such that if $n\ge n_0$ and $G$ is an $n$-vertex $F$-free graph, then
\begin{equation}  \label{eq4-Niki}
 \lambda (G)\le \lambda (T_{n,r}),
 \end{equation}
where the equality holds if and only if $G$  is the Tur\'{a}n graph $T_{n,r}$.
It is worth noting that
Nikiforov's result (\ref{eq4-Niki}) can imply Simonovits' result (\ref{eq-Sim}) except for the uniqueness;
see, e.g., \cite[Fact 1.1]{ZL2022jgt}.

\medskip

The study on counting substructures and eigenvalues was initiated by
Bollob\'{a}s and Nikiforov \cite{BN2007jctb} in 2007, who
provided a number of relations between the spectral radius
and the number of cliques for a graph $G$ with given order $n$ and size $m$.
In particular, the base case gives a spectral
supersaturation on triangles, which asserts that
if $G$ is a graph with the spectral radius $\lambda $, then
\begin{equation*} \label{eq-spectral-super-triangle}
  t(G) \ge \frac{\lambda \bigl(\lambda^2 - m\bigr)}{3}
  \end{equation*}
  and
  \begin{equation} \label{eqeq-BN2}
  t(G) > \frac{n^2}{12}\left( \lambda - \frac{n}{2}\right). 
  \end{equation}
Additionally, the first inequality
was  independently proved by Cioab\u{a},  Feng,
Tait and Zhang \cite{CFTZ20},
and it also studied by Ning and Zhai \cite{NZ2021}.
These alternative proofs are quite different.
Moreover,  Ning and Zhai \cite{NZ2021} proved  a spectral version of (\ref{eq-Erdos-1})
by showing that
if $ \lambda (G) \ge \lambda (T_{n,2}) $,
then
\begin{equation} \label{thmNZ2021}
t(G) \ge \left\lfloor \frac{n}{2} \right\rfloor -1 ,
\end{equation}
unless $G=T_{n,2}$.
Furthermore,
an even more extensive problem, first explicitly posed by
Ning and Zhai \cite{NZ2021}, is to study the spectral supersaturation
behind Mubayi's theorem.

\begin{problem}[Ning--Zhai \cite{NZ2021}, 2023]   \label{prob-NZ} 
\quad \\ 
(i) (The general case)
Find a spectral version of Mubayi's result.  \\
(ii) (The critical case) For $q=1$, where $q$ is defined as
in Theorem \ref{thm-Mubayi}, find the tight spectral versions when $F$ is some particular color-critical graph, such as triangle, clique, book, odd cycle
or even wheel, etc.
\end{problem}

In this paper, we focus our attention on Problem \ref{prob-NZ}
and intend to establish a spectral version in
the general case for color-critical graphs.
Unlike the size $e(G)$ of a graph $G$,
it seems not obvious to
measure the increment on the spectral radius $\lambda (G)$,
since $e(G)$ is always an integer, while
$\lambda (G)$ is not.
As commented by Ning and Zhai on the footnote
in \cite{NZ2021},
finding a spectral version of Mubayi's result,
one had better to give an estimate or a formula of
spectral correspondence of $c(n,F)$.
Our contribution in this paper is twofold.
Firstly, we shall give a new way to measure the increment on the spectral radius
and present a spectral conjecture of Theorem \ref{thm-Mubayi}.
Secondly, we will confirm the triangle case by showing a spectral version of Theorem \ref{thmrad}.

\section{Main results}

\subsection{A spectral conjecture on Mubayi's result}

In what follows, we shall establish a
spectral condition  corresponding to
the edge condition $e(G)=e(T_{n,r}) +q$. 
 We propose the following counting problem,
which provides a spectral version on Mubayi's result.
Recall that $c(n,F)$ is the minimum number of copies of $F$ in the graph obtained from
the $r$-partite Tur\'{a}n graph
$T_{n,r}$ by adding exactly one edge.
In particular, setting $F=C_3$ or $C_5$, then $r=2$, we have 
$c(n,C_3)= \lfloor \frac{n}{2}\rfloor$ and $c(n,C_5)=
\lfloor \frac{n}{2}\rfloor (\lfloor \frac{n}{2}\rfloor -1) 
(\lceil \frac{n}{2}\rceil -2)$.
Let $T_{n,r,q}$
be the $n$-vertex graph obtained from Tur\'{a}n's graph
$T_{n,r}$
by embedding a star $K_{1,q}$ with
$q$ edges into a vertex part of size $\lceil \frac{n}{r} \rceil$.
We propose the following conjecture.

\begin{conjecture}  \label{conj-main}
Let  $F$ be a color-critical graph
with $\chi (F)=r+1 \ge 3$. There exists $\delta =\delta (F)>0$ such that
if $n$ is sufficiently large, $1\le q < \delta n$, and
$G$ is an $n$-vertex graph with
\[  \lambda (G)\ge \lambda (T_{n,r,q}),\]
then $G$ contains at least $q\cdot c(n,F) $ copies of $F$.
Furthermore, the graph $T_{n,r,q}$ is the unique spectral extremal graph attaining the minimum number of copies of $F$.
\end{conjecture}

Here, we would like to
give an explanation of the original ideas of proposing such a conjecture.  Recently,
Pikhurko and Yilma \cite{PY2017} extended Mubayi's result and proved that
there exist $\delta =\delta (F)>0$
and $n_0$ such that
for all $n\ge n_0$  and $1\le q < \delta n$,
if $G$ is an $n$-vertex graph
with $e(T_{n,r}) +q$ edges which contains the smallest
number of copies of
$F$, then $G$ is obtained from $T_{n,r}$ by adding $q$ edges.
More precisely, the extremal graphs with minimum copies of $F$ are obtained from $T_{n,r}$ by embedding
a small subgraph with $q$ edges into a larger vertex part.
Upon computation, we find that
the resulting graph attains the maximum spectral radius
when we choose the star $K_{1,q}$  as the embedded subgraph.

\medskip
Benefited from the  techniques of Ning and Zhai \cite{NZ2021},
we are able to confirm  Conjecture \ref{conj-main}
 in the base case $q=1$ by counting the copies of triangles
 for all integers $n$,
not only for sufficiently large $n$.
 Moreover, we show further that the spectral extremal graph is unique.
 Our result can be viewed as a spectral version of Erd\H{o}s--Rademacher's result in Theorem \ref{thmrad}.
In the sequel, we denote by $K_{\lceil \frac{n}{2} \rceil, \lfloor \frac{n}{2} \rfloor}^+$ the graph obtained from the balanced complete bipartite graph
 $K_{\lceil \frac{n}{2} \rceil, \lfloor \frac{n}{2} \rfloor}$ by
adding exactly one edge into the
part of larger size; see Figure \ref{fig-Erdos-Rade}.

 \begin{figure}[H]
\centering
\includegraphics[scale=0.85]{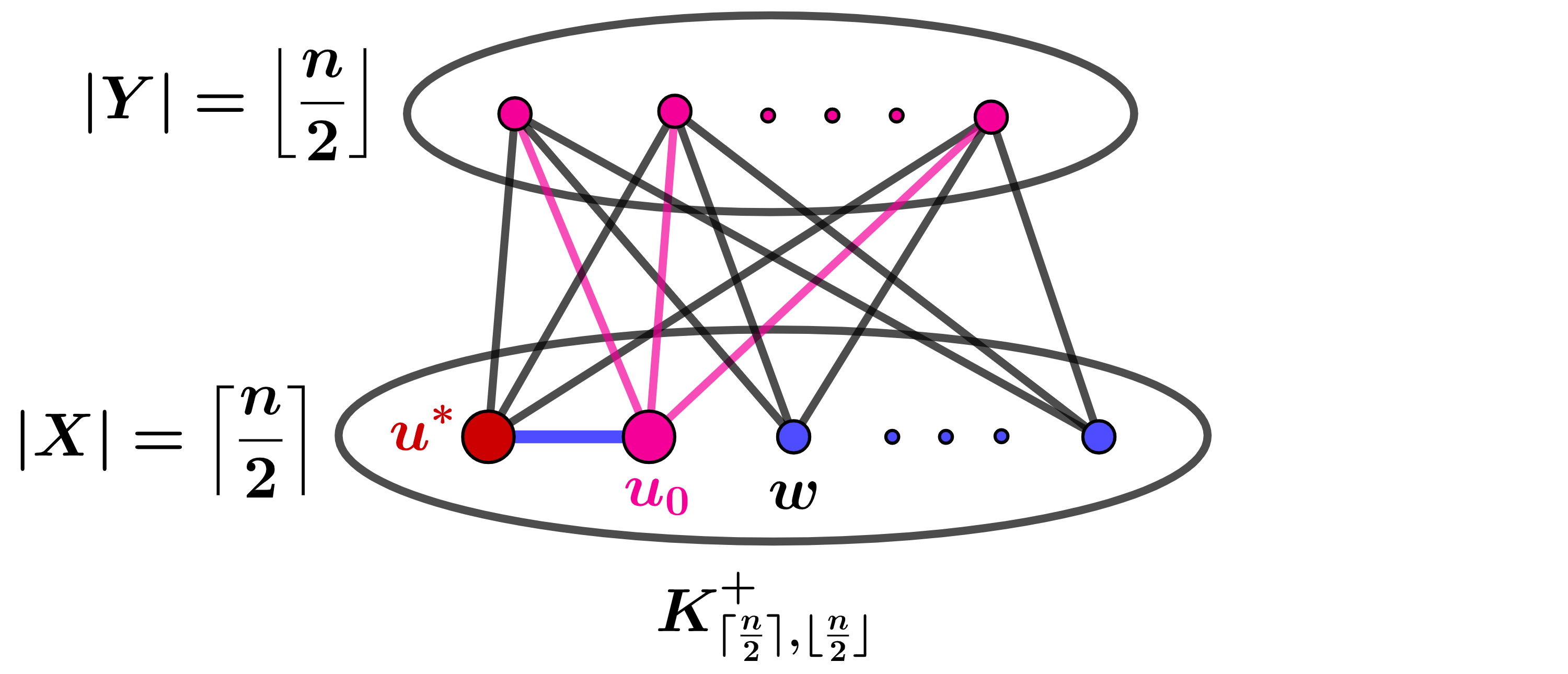}
\caption{The unique extremal graph in Theorem \ref{thm-LLP}.}
\label{fig-Erdos-Rade}
\end{figure}

\begin{theorem} \label{thm-LLP}
If $G$ is a graph on $n$ vertices with
\[  \lambda (G) \ge \lambda (K_{\lceil \frac{n}{2} \rceil, \lfloor \frac{n}{2} \rfloor}^+) ,\]
then $ t(G) \ge \lfloor \frac{n}{2}\rfloor$,
with equality  if and only if
$G=K_{\lceil \frac{n}{2} \rceil, \lfloor \frac{n}{2} \rfloor}^+$.
\end{theorem}

This paper is organized as follows.
In Section \ref{sec2}, we will provide some preliminaries
and lemmas for our purpose, and list the sketch of the proof
of Theorem \ref{thm-LLP}.
In Section \ref{sec3}, we shall put our attention to the proof of
Theorem \ref{thm-LLP}.
The  techniques of our proof of Theorem \ref{thm-LLP}
are motivated by  that in \cite{NZ2021}.
To some extent, the results in this paper
could be regarded as a  complement of the works in \cite{NZ2021,NZ2021b}.
In our setting,
we need to make more structural analysis of $G$ since the uniqueness is determined.
Indispensably, some different treatments and techniques are needed in our proof.
Finally, we give some concluding remarks and open problems in Section \ref{sec5}.

\subsection{Remarks on  Theorem \ref{thm-LLP}}

Under the condition of Theorem \ref{thm-LLP},
Bollob\'{a}s--Nikiforov's result (\ref{eqeq-BN2}) can also gives
a bound on $t(G)$. Indeed, using a forthcoming Lemma \ref{lem-21},
one can obtain that if $n$ is even, then $\lambda (G) > \frac{n}{2} + \frac{2}{n}$.
Applying (\ref{eqeq-BN2}), we obtain $t(G) > \frac{n}{6}$;
if $n$ is odd, then $\lambda (G) > \frac{n}{2} + \frac{1.75}{n}$
 and $t(G) >0.1458n$.
 On the other hand, using Lemma \ref{lem-21} again,
 we can get $\lambda (K_{\lceil \frac{n}{2} \rceil, \lfloor \frac{n}{2} \rfloor}^+) <  \frac{n}{2} + \frac{2.1}{n}$ for large $n$.
 Then the lower bound on $t(G)$ guaranteed by
  (\ref{eqeq-BN2})  is no more than $0.175n$.
Thus, our result in Theorem \ref{thm-LLP} provides a stronger bound
 $t(G) \ge 0.5n$ in this sense.
More generally, for a graph $G$ satisfying $\lambda (G) > \lambda (T_{n,r})$ and $\lambda (G)$
is very close to $\lambda (T_{n,r})$,
 Conjecture \ref{conj-main}
also provides a better bound on the number of cliques $K_{r+1}$ than
that of Bollob\'{a}s and Nikiforov  \cite{BN2007jctb}. 
Although Theorem \ref{thm-LLP} provides a corresponding spectral version of
Theorem \ref{thmrad},
it is not a generalization on spectral Mantel's theorem.
Moreover,  Theorem \ref{thm-LLP}
and the bound in (\ref{thmNZ2021}) are not comparable
since both the condition and result of
Theorem \ref{thm-LLP} are slightly stricter.
Furthermore,
the spectral extremal graph of Theorem \ref{thm-LLP} is determined uniquely,
while the equality case in (\ref{thmNZ2021})  is not.  

\medskip
Last but not least, Erd\H{o}s--Rademacher's theorem requires 
$e(G) > e(T_{n,2})$. Intuitively, 
one may ask in Theorem \ref{thm-LLP}  that
if $G$  satisfies  a weaker condition, namely,
$ \lambda (G)> \lambda (T_{n,2})$,
 whether $G$ contains at least  $\lfloor \frac{n}{2}\rfloor$ triangles.
 Unfortunately, this result does not hold.
 Indeed,
when $n$ is even, we denote by $K_{\frac{n}{2}+1,\frac{n}{2}-1}^+$
 the graph obtained from the bipartite graph $K_{\frac{n}{2}+1,\frac{n}{2}-1}$ by adding an edge to
the vertex part of size $\frac{n}{2}+1$.
Upon computation, it follows that $\lambda (K_{\frac{n}{2}+1,\frac{n}{2}-1}^+) $ is
the largest root of the polynomial
\[ f_1(x)= x^3 - x^2 + (1- \tfrac{n^2}{4})x + \tfrac{n^2}{4} - n + 1 .\]
Since $f_1(\frac{n}{2})= 1-\frac{n}{2} <0$ for $n\ge 4$,
we have 
\[  \lambda (K_{\frac{n}{2}+1,\frac{n}{2}-1}^+) > \lambda (T_{n,2})=\frac{n}{2}.\] 
However,  the number of triangles in $K_{\frac{n}{2}+1,\frac{n}{2}-1}^+$ is exactly $\frac{n}{2}-1$.
This example demonstrates that
$t(G)\ge \lfloor \frac{n}{2}\rfloor$ can  not be guaranteed
in a graph $G$ with  $\lambda (G) > \lambda (T_{n,2})$.

\medskip
Upon calculations, one can check 
the following graphs that satisfy
 $\lambda (G) > \lambda (T_{n,2})$ and
 $t(G)= \left\lfloor \frac{n}{2}\right\rfloor -1$.
For simplicity, we provide these graphs without details.

\medskip
(a)  For even integer $n$,
as we have discussed above,
the graph $G=K_{\frac{n}{2}+1, \frac{n}{2}-1}^+$ is an
extremal graph with $t(G)= \frac{n}{2}  -1$.
In addition, there is one more possibility for even $n$. We denote
by $G=K_{\frac{n}{2} ,  \frac{n}{2} }^{+|}$  the graph obtained from
$K_{ \frac{n}{2},  \frac{n}{2}}$ by adding an edge $e_1$ into the part of size $ \frac{n}{2} $ and deleting an edge $e_2$ between two parts such that $e_2$ is incident to $e_1$.
By calculation, we obtain that $\lambda (K_{\frac{n}{2} ,  \frac{n}{2} }^{+|} )$
is the largest root of
\[  f_2(x)=
x^4 - \tfrac{n^2}{4} x^2 - (n-2)x +1+ \tfrac{n^2}{2} -2n. \]
One can check that $f_2(\frac{n}{2})=1-n<0$ and hence $\lambda (K_{\frac{n}{2} ,  \frac{n}{2} }^{+|} )>\frac{n}{2} =
\lambda (T_{n,2})$.

(b) For odd integer $n$, we take
$G=K_{\frac{n+1}{2} , \frac{n-1}{2}}^{+|}$,  which is the graph obtained from
$K_{ \frac{n+1}{2},  \frac{n-1}{2}}$ by adding an edge $e_1$ into the part of size $ \frac{n+1}{2} $ and deleting an edge $e_2$ between two parts such that $e_2$ is incident to $e_1$.
By a direct calculation, we know that $\lambda
(K_{ \frac{n+1}{2}, \frac{n-1}{2}}^{+|} )$
is the largest root of
\[  f_3(x)=x^4 -\tfrac{n^2-1}{4}x^2 -
(n-3)x + \tfrac{n^2}{2} - 2n+\tfrac{3}{2}. \]
 Moreover, one can verify that
$f_3(\frac{\sqrt{n^2-1}}{2})=\frac{n-3}{2}(n-1-\sqrt{n^2-1})<0$,
which implies $\lambda (K_{ \frac{n+1}{2},\frac{n-1}{2}}^{+|})
> \frac{\sqrt{n^2-1}}{2} = \lambda (T_{n,2})$.
It is easy to observe that $t(K_{ \frac{n+1}{2},\frac{n-1}{2}}^{+|}) = \frac{n-1}{2} -1$.

\section{Lemmas and outline of the proof}

\label{sec2}

Before stating  the detailed proof of Theorem \ref{thm-LLP},
we need to present a lemma, which provides 
a characterization of the spectral radius of the graph
$K_{ \lceil \frac{n}{2} \rceil, \lfloor \frac{n}{2} \rfloor}^+$.

\begin{lemma} \label{lem-21}
(a) If $n$ is even, then $\lambda(K_{\frac{n}{2}, \frac{n}{2}}^+)$
is the largest root of
\[  f(x)=x^3-x^2 - \frac{n^2}{4}x + \frac{n^2}{4}-n. \]
(b) If $n$ is odd, then $\lambda(K_{\frac{n+1}{2}, \frac{n-1}{2}}^+)$
is the largest root of
\[ g(x)=x^3 - x^2 + \frac{1-n^2}{4} x +
\frac{n^2}{4} - n + \frac{3}{4}. \]
Consequently, we have
\begin{equation} \label{eq-n2+2}
  \lambda^2 (K_{ \lceil \frac{n}{2} \rceil, \lfloor \frac{n}{2} \rfloor}^+) >\left\lfloor \frac{n^2}{4} \right\rfloor +2.
  \end{equation}
\end{lemma}

\begin{proof}
(a) Let $X=(x_1,x_2,\ldots ,x_n)^T$ be a Perron eigenvector corresponding to $\lambda(K_{\frac{n}{2}, \frac{n}{2}}^+)$.
We partition the vertex set of $ K_{\frac{n}{2}, \frac{n}{2}}^+$ as $\Pi$:
\[   V(K_{\frac{n}{2}, \frac{n}{2}}^+ )=
X_1 \cup X_2 \cup Y, \]
 where $X_1=\{u^*,u_0\}$ forms an edge,
$X_1\cup X_2$ and $Y$ are partite sets of $K_{\frac{n}{2}, \frac{n}{2}}$.
By comparing the neighborhoods, we can see that $x_{u^*}=x_{u_0}$,
all coordinates of the vector $X$
corresponding to vertices of $X_2$  are equal
(the coordinates of vertices of $Y$ are equal).
Without loss of generality, we may assume that
$x_{u^*}=x_{u_0}=x$, $x_w=y$ for each $w\in X_2$, and $x_v=z$
for each $v\in Y$. Then
\[  \begin{cases}
\lambda x= x + \frac{n}{2} z, \\
\lambda y = \frac{n}{2}z, \\
\lambda z = 2x + (\frac{n}{2}-2)y.
\end{cases} \]
Thus $\lambda(K_{\frac{n}{2}, \frac{n}{2}}^+)$ is the largest eigenvalue of
\[  B_{\Pi} =  \begin{bmatrix}
1 & 0 & \frac{n}{2} \\
0 & 0 & \frac{n}{2} \\
2 & \frac{n}{2}-2 & 0
\end{bmatrix}.  \]
By calculation, we know that $\lambda(K_{\frac{n}{2}, \frac{n}{2}}^+)$ is the largest  root of
\[ f(x)= \det (xI_3 - B_{\Pi})=x^3-x^2 - \frac{n^2}{4}x + \frac{n^2}{4}-n.  \]
(b) For odd $n$, the proof is similar with the previous case.
We partition the vertex set of $K_{\frac{n+1}{2}, \frac{n-1}{2}}^+$
as $\Pi'$:
\[  V(K_{\frac{n+1}{2}, \frac{n-1}{2}}^+) = X_1\cup X_2\cup Y, \]
where $X_1=\{u^*,u_0\}$ forms an edge, $X_1\cup X_2$ and $Y$
are partite sets of $K_{\frac{n+1}{2}, \frac{n-1}{2}}$
satisfying $|X_1| + |X_2|= \frac{n+1}{2}$
and $|Y|=\frac{n-1}{2}$. Then
a similar argument yields that
$\lambda (K_{\frac{n+1}{2}, \frac{n-1}{2}}^+)$ is the largest eigenvalue of
\[  B_{\Pi'}= \begin{bmatrix}
1 & 0 & \frac{n-1}{2} \\
0 & 0 & \frac{n-1}{2} \\
2 & \frac{n+1}{2}-2 & 0
\end{bmatrix}.  \]
Thus, $\lambda (K_{\frac{n+1}{2}, \frac{n-1}{2}}^+)$
is the largest root of
\[  g(x)= \det (xI_3- B_{\Pi'})
=x^3 - x^2 + \frac{1-n^2}{4} x +
\frac{n^2}{4} - n + \frac{3}{4}.  \]
Finally, we are ready to prove (\ref{eq-n2+2}).
By calculation, we can verify that for even $n$,
\[  f \left(\sqrt{{n^2}/{4 }+2} \right) = \sqrt{n^2+8} -n -2<0 \]
and for every odd $n$,
\[  g \left(\sqrt{{(n^2-1)}/{4} +2} \right) = \sqrt{n^2+7} -n-1 <0. \]
This completes the proof.
\end{proof}

Rayleigh's formula also gives a lower bound
on the spectral radius.
\begin{equation} \label{eq-3}
\lambda (K_{ \lceil \frac{n}{2} \rceil,  \lfloor \frac{n}{2} \rfloor}^+)
> \frac{2e(K_{ \lceil \frac{n}{2} \rceil,  \lfloor \frac{n}{2} \rfloor}^+)}{n}
= \begin{cases}
\frac{n}{2} + \frac{2}{n} & \text{if $n$ is even;}\\
\frac{n}{2} + \frac{3}{2n} & \text{if $n$ is odd.}
\end{cases}
\end{equation}
In particular, if $n$ is even, then we get from (\ref{eq-3})
that $\lambda^2 > (\frac{n}{2} + \frac{2}{n})^2 >
\frac{n^2}{4} +2$.
Thus, this observation also provides an alternative proof of (\ref{eq-n2+2}) for even $n$.

\medskip
The following lemma \cite{WXH2005} is also needed in our proof,
it provides an operation of a connected graph which increases the adjacency spectral radius strictly.

\begin{lemma}[Wu--Xiao--Hong \cite{WXH2005}, 2005] \label{lem-22}
Let $G$ be a connected graph
and $(x_1,\ldots ,x_n)^T$ be a Perron vector of $G$,
where $x_i$ corresponds to $v_i$.
Assume that
 $v_i,v_j \in V(G)$ are vertices such that $x_i \ge x_j$, and $S\subseteq N_G(v_j) \setminus N_G(v_i)$ is non-empty.
 Denote $G^*=G- \{v_jv : v\in S\} +
\{v_iv : v\in S\}$. Then $\lambda (G) < \lambda (G^*)$.
\end{lemma}

\medskip
The main steps of the proof of Theorem \ref{thm-LLP} can be outlined as below.

\medskip
First of all, suppose that $G$ is a graph  on $n$ vertices with
the spectral radius $\lambda(G) \ge \lambda (K_{\lceil \frac{n}{2} \rceil, \lfloor \frac{n}{2} \rfloor}^+)$ and
the number of triangles $t(G)\le \lfloor \frac{n}{2}\rfloor $,
our aim is to characterize the structure of $G$ and
show that $G$ is the required extremal graph $K_{\lceil \frac{n}{2} \rceil, \lfloor \frac{n}{2} \rfloor}^+$.
Let $\bm{x}=(x_1,\ldots ,x_n)^T$ be the unit Perron vector of $G$, and
choose $u^*\in V(G)$ as a vertex such that $x_{u^*}=\max_{u\in V(G)} x_u$. Let $t^*$ be the number of triangles in $G$ containing the vertex $u^*$. Denote by $A=N_G(u^*)$ the set of neighbors of $u^*$, and $B$ the set of remaining vertices outside $A\cup \{u^*\}$; see Figure \ref{fig-Erdos-Rade}.

\begin{itemize}
\item[$\spadesuit$]
Firstly, applying Lemma \ref{lem-21},
we will show that if $n$ is even, then $|A| \ge  \frac{n}{2} +1$.
If $n$ is odd, then $|A| \ge \frac{n+1}{2}$; see  the forthcoming Claim 1. Moreover, we shall show that $t^*=e(A)\ge 1$ and $|B|\ge 1$, and there is no vertex of $B$ which is adjacent to all vertices of $A$; see Claims 2 and 3.

\item[$\heartsuit$]
Secondly, we shall prove that $t=t^*$. In other words,
all triangles in $G$ must contain the vertex $u^*$; see Claim 4.
By Lemma \ref{lem-22},
we shall prove further that the induced subgraph
$G[A] = K_{1,t^*} \cup (|A| - t^* -1)K_1$; see Claim 5.

\item[$\clubsuit$]
Thirdly, we denote by $u_0$
the central vertex of the star $K_{1,t^*}$ in $G[A]$.
Then we will show that $d_B(u_0) =0$; see Claim 6.
Moreover, we will prove that $B$ is an independent set in $G$,
and then $N_G(w)=A\setminus \{u_0\}$ for each $w\in B$;
see Claims 7 and 8.

\item[$\diamondsuit$]
Finally, denote by $A_0$ the set of isolated vertices of $G[A]$,
we will  show that
$\sum_{u\in A_0} x_u < x_{u^*}$; see Claim 9.
Furthermore, we will finish the proof by showing that $|A|=t+1$
and $A_0=\varnothing$; see Claim 10.
Consequently, $G$ must be the extremal graph 
$K_{\lceil \frac{n}{2} \rceil, \lfloor \frac{n}{2} \rfloor}^+$.
\end{itemize}

Generally speaking, the above steps introduced the key ideas of the approach of this paper
for treating the problem on counting triangles.
In next section, we shall provide the detailed proofs.

\section{Proof of Theorem \ref{thm-LLP}}

\label{sec3}

Now, we are ready to prove Theorem \ref{thm-LLP}.

\begin{proof}[{\bf Proof of Theorem \ref{thm-LLP}}]
Assume that $G$ is a graph  on $n$ vertices with
 $\lambda(G) \ge \lambda (K_{\lceil \frac{n}{2} \rceil, \lfloor \frac{n}{2} \rfloor}^+)$ and
 $t(G)\le \lfloor \frac{n}{2}\rfloor $,
 our goal is to prove that $G=K_{\lceil \frac{n}{2} \rceil, \lfloor \frac{n}{2} \rfloor}^+$.
Moreover,  we may assume further that
$\lambda (G)$ is the maximum. Then $G$ is connected. 
Otherwise, adding an edge between two components can increase
the spectral radius and preserve the number of triangles.
Let $\bm{x}=(x_1,x_2,\ldots ,x_n)^T$ be the unit Perron vector of $G$.
Furthermore, we may assume that $u^*\in V(G)$ is a vertex such that $x_{u^*}=\max_{u\in V(G)} x_u$.
For convenience, we denote $\lambda = \lambda (G) $ and $t=t(G)$.
Let $t^*$ be the number of triangles containing $u^*$ in $G$. Let $A=N_G(u^*)$ and $B=V(G) \setminus (N_G(u^*) \cup \{u^*\})$. Clearly, we have $t^*=e(A)$.
In what follows, we shall present the proof step by step and show several key claims.

\medskip
{\bf Claim 1.}
If $n$ is even, then $|A| \ge  \frac{n}{2} +1$.
If $n$ is odd, then $|A| \ge \frac{n+1}{2}$.
Consequently, we get
 $|A| \ge \lceil \frac{n}{2}\rceil$.
Combining with $|A| + |B| +1=n$, we have
 $|B|\le \lfloor \frac{n}{2}\rfloor -1$.

\begin{proof}
By the maximality of $x_{u^*}$, we can get
$\lambda x_{u^*}= \sum_{u\in A} x_u \le |A| x_{u^*}$,
which leads to $ \lambda \le |A|$.
By Lemma \ref{lem-21}, we know that
$\lambda^2 >
\lfloor \frac{n^2}{4} \rfloor +2$.
If $n$ is even, then $|A| \ge \lambda > \frac{n}{2}$
and so $|A| \ge \lceil \frac{n}{2} \rceil +1$.
If $n$ is odd, then $|A| \ge \lambda >
(\frac{1}{4}(n^2-1) +2)^{1/2} >
\frac{n-1}{2}$. Since $|A|$ is an integer,
we get $|A| \ge \frac{n-1}{2} +1
=\lceil \frac{n}{2} \rceil$.
\end{proof}

\medskip
{\bf Claim 2.}
$t^* =e(A)\ge 1$ and $|B|\ge 1$.

\begin{proof}
Note that $\lambda x_{u^*}= \sum_{u\in A} x_u$. Then
\begin{equation} \label{eqeq-6}
\lambda^2 x_{u^*} = \sum_{u\in A}
\sum_{w\in N(u)} x_w =
|A| x_{u^*} + \sum_{u\in A} d_A(u) x_u + \sum_{w\in B}
d_A(w) x_w.
\end{equation}
Recall that $t^*=e(A)$. Assume on the contrary that $t^*=0$. Then $d_A(u)=0$ for each $u\in A$.
Since $\sum_{w\in B} d_A(w)x_w \le
e(A,B)x_{u^*} \le |A| |B|x_{u^*}$. Thus (\ref{eqeq-6}) gives
\[  \lambda^2 x_{u^*}  \le
 |A| x_{u^*}  + \sum_{w\in B} d_A(w) x_w
 \le |A|(|B| +1) x_{u^*} \le
 \left\lfloor  \frac{n^2}{4} \right\rfloor x_{u^*}. \]
It follows that
$\lambda^2 \le \lfloor \frac{n^2}{4}\rfloor$, which leads to
a contradiction with Lemma \ref{lem-21}.
Thus, we get $t^*\ge 1$. If $|B|=0$, then
$V(G)=\{u^*\} \cup A$ and $|A| =n-1$.
Using (\ref{eqeq-6}) and $t^*\le t \le \lfloor \frac{n}{2}\rfloor$,
we have $\lambda^2 \le |A| +2t^*
\le n-1 + 2\lfloor \frac{n}{2}\rfloor \le \lfloor \frac{n^2}{4}\rfloor +2$, where the last inequality holds for every $n\ge 3$ and $n\neq 4$.
This is a contradiction as Lemma \ref{lem-21}.
For the case $n=4$, we have $|B|=3$ and $t^*\in \{1, 2\}$.
Then either $G=K_1\vee (K_2\cup K_1)$ or $G=K_1\vee P_3 = K_{2,2}^+$.
Since $\lambda (G) \ge \lambda (K_{2,2}^+)$,
we get $G=K_{2,2}^+$, which is the desired extremal graph.
\end{proof}

In fact, the key idea of the remaining proof
depends on the full use of (\ref{eqeq-6}).
In particular, we will analyse the structure of $G$ to
refine the bound on $e(A,B)$.
Let $B_i$ be the set of vertices of $B$ which has
exactly $i$ non-neighbors in $A$, that is,
\[  B_i:=\{w\in B: d_A(w) = |A| -i\}. \]
Denote $b_i=|B_i|$ for $i\in \{0,1,2,\ldots ,|A|\}$. Then
\begin{equation}  \label{eqeq-7}
\begin{aligned}
e(A,B) &\le
(|A| -2) (|B| -b_0-b_1) +
(|A| -1) b_1 + |A| b_0  \\
& = (|A| -2) |B| +2b_0 +b_1.
\end{aligned}
\end{equation}

\medskip
{\bf Claim 3.}  $b_0=0$.

\begin{proof}
For every $w\in B_0$,
there are $t^*$ triangles consisting of $w$
and an edge of $A$.
On the other hand, each edge of $A$, together with
the vertex $u^*$ forms a triangle.
Thus, we have
$b_0t^* + t^* \le t$, and  then
$b_0 \le \frac{t}{t^*} -1$.
Suppose on the contrary that
$b_0\neq 0$, and so $b_0\ge 1$.
Then we have $1\le t^* \le \frac{t}{2}$ and $t\ge 2$.
Note that $b_0+b_1\le |B|$ and $b_0 \le \frac{t}{t^*} -1$.
It follows that $2b_0+b_1 \le |B| + \frac{t}{t^*} -1$.
By (\ref{eqeq-7}), we obtain
\[  \sum_{w\in B} d_A(w) x_w \le e(A,B) x_{u^*}
\le \left(|A||B| - |B| + \frac{t}{t^*} -1 \right)x_{u^*}. \]
Moreover, one can observe that
\begin{equation} \label{eq-2eA}
\sum_{u\in A} d_A(u)x_u \le 2e(A)x_{u^*} =2t^*x_{u^*}.
\end{equation}
Thus, we get from (\ref{eqeq-6}) and (\ref{eq-2eA})  that
\begin{equation} \label{eqeq-8}
 \lambda^2x_{u^*} \le
\left( |A| + |A||B| - |B| + \frac{t}{t^*} -1 +2t^* \right)x_{u^*}.  \end{equation}
Denote $g(|B|)=|A| +|A| |B| - |B|$
and $f(t^*)=\frac{t}{t^*} -1 +2t^*$.
Recall that $|A|=n-1-|B|$,
by computation, we know that
$g(|B|)$ attains the maximum at $|B| = \lfloor \frac{n}{2}\rfloor -1$. Thus, we get
$g(|B|) \le g(\lfloor \frac{n}{2}\rfloor -1)
= \lfloor \frac{n^2}{4}\rfloor - \lfloor \frac{n}{2}\rfloor +1$.
On the other hand, observe that
$f(t^*)$ is convex on $t^* \in [1,\frac{t}{2}]$.
Then $f(t^*) \le \max\{ f(1), f(\frac{t}{2})\}
= t+1 \le  \lfloor \frac{n}{2}\rfloor +1$.
Therefore, we get from (\ref{eqeq-8}) that
$\lambda^2 \le \lfloor \frac{n^2}{4}\rfloor +2$,
which contradicts with $\lambda^2> \lfloor \frac{n^2}{4}\rfloor +2$ by
Lemma \ref{lem-21}.
Hence, we get $b_0=0$.
\end{proof}

{\bf Claim 4.} We have $t=t^*$, that is,
all triangles in $G$ contain the vertex $u^*$.

\begin{proof}

Assume on the contrary that $t^*\le t-1$.
To obtain a contradiction,
 we are going to prove the following
three items, respectively.

\begin{itemize}

\item[(a)] $G[A] = K_{1,t^*} \cup (|A| - t^* -1)K_1$.

\item[(b)] $e(B)=0$, that is, $B$ is an independent set in $G$.

\item[(c)] Let $u_0$ be the central vertex of the star $K_{1,t^*}$ in $G[A]$.
Then $d_B(u_0) =0$.
\end{itemize}

\begin{proof}
For (a), suppose on the contrary that
$G[A]\neq K_{1,t^*} \cup (|A| - t^* -1)K_1$.
Then for each vertex $w\in B_1$, i.e.,
$d_A(w)= |A|-1$, the non-adjacent vertices of $w$
in $A$ is a single vertex,
there is at least one edge  $\{u,v\} \in E(G[A])$
such that  $\{u,v\} \subseteq N_G(w) \cap A$.
This implies that there are at least $b_1$ triangles
consisting of a vertex of $B_1$ and an edge of $G[A]$.
Thus, we obtain $b_1 + t^* \le t$.
Combining (\ref{eqeq-7}) with Claim 3, we get
\[  e(A,B) \le (|A|-2) |B| + b_1 \le  (|A|-2) |B|  + t-t^*.\]
Using (\ref{eqeq-6})  and (\ref{eq-2eA}) again, we have
\begin{equation} \label{eqeq-9}
\lambda^2 x_{u^*} \le
|A| x_{u^*} + 2t^*x_{u^*} +e(A,B)x_{u^*}
\le (|A| + |A||B| - 2|B| + t + t^*)x_{u^*}.
\end{equation}
By computation, we know that
the maximum of $|A| + |A||B| - 2|B|$
attains at $|B|=\lfloor \frac{n}{2}\rfloor -2$
since $|A| +|B| + 1=n$.
Then we get $|A| + |A||B| - 2|B| \le
\lfloor \frac{n^2}{4}\rfloor
-2 \lfloor \frac{n}{2}\rfloor +3$.
Moreover, recall that  $t\le \lfloor \frac{n}{2}\rfloor$, and
then $t^*\le t-1 \le \lfloor \frac{n}{2}\rfloor -1$. Therefore, we get  by (\ref{eqeq-9}) that
$ \lambda^2 \le  \lfloor \frac{n^2}{4}\rfloor +2 $, which is a contradiction by Lemma \ref{lem-21}.

For (b), suppose on the contrary that
there exists an edge $\{w_i,w_j\} \in E(G[B])$,
say $w_i \in B_i$ and $w_j \in B_j$.
By using (\ref{eqeq-6}) and (\ref{eq-2eA}), we have
\begin{equation} \label{eqeq-10}
  \lambda^2\le |A| + 2t^* + e(A,B).
  \end{equation}
Since $b_0=0$ by Claim 3,
we get  that
$e(A,B) \le |A| |B| - (|B|-2) -i-j $.
Note that $d_A(w_i)= |A| -i$
and $d_A(w_j)= |A| -j$. Then
$w_i,w_j$ have at least $d_A(w_i) + d_A(w_j) - |A| $
common neighbors. Thus there are
$|A| -i-j$ triangles containing $w_i$ and $w_j$,
and we have $t^* + (|A| -i-j) \le t$,
which together with (\ref{eqeq-10})
yields
\begin{equation*}
\begin{aligned}
 \lambda^2 &\le
|A| + 2t^* + |A||B| -|B| +2- i-j   \\
& \le t^* +t + |A||B| - |B| +2.
\end{aligned}
\end{equation*}
Similarly, we have
$t^*+t\le 2t-1 \le 2\lfloor \frac{n}{2}\rfloor -1$.
Since
$|A| +|B|+1=n$ and $|A||B| - |B|$ attains the maximum at
$|B|=\lfloor \frac{n}{2}\rfloor -1$.
Then we get $(|A| -1) |B| \le
(\lceil \frac{n}{2}\rceil -1)(\lfloor \frac{n}{2}\rfloor -1)
\le \lfloor \frac{n^2}{4}\rfloor -n +1$.
Thus, we have
$ \lambda^2 \le  \lfloor \frac{n^2}{4}\rfloor -n +1 +
2\lfloor \frac{n}{2}\rfloor -1 +2 \le  \lfloor \frac{n^2}{4}\rfloor +2  $, a contradiction.

For (c), assume on the contrary that
there exists an integer
$i\in \{1,2,\ldots ,|A|\}$ and $w_i \in B_i$ such that
$u_0w_i$ is an edge in $G[B]$.
We know from  Claim 3 that $b_0=0$.
Since $d_A(w_i) = |A| -i$,
we obtain
 \begin{equation} \label{eqeq-11}
 e(A,B) \le (|A| -1)|B| - (i-1).
 \end{equation}
If $i\le t^*$, then by applying Part (a),
there are at least $t^*-i$ triangles
consisting of $u_0,w_i$ and a leaf vertex of $K_{1,t^*}$.
Thus, we have $t^* + (t^* -i) \le t$,
and so $-i\le t-2t^*$. Then
$e(A,B) \le (|A| -1)|B| +1 +t-2t^*$.
By (\ref{eqeq-10}), we have
\[  \lambda^2 \le |A| +(|A| -1)|B| +1 +t.   \]
Since $t\le \lfloor \frac{n}{2}\rfloor$ and
$|A| + (|A| -1) |B|$ attains the maximum
at $|B|=\lfloor \frac{n}{2}\rfloor -1$.
Therefore, we get $\lambda^2
\le \lceil \frac{n}{2}\rceil + (\lceil \frac{n}{2}\rceil  -1)
(\lfloor \frac{n}{2}\rfloor -1) +1 + \lfloor \frac{n}{2}\rfloor
\le \lfloor \frac{n^2}{4}\rfloor +2$, a contradiction.
Therefore, we must have $i\ge t^* +1$, that is,
$-i +1\le -t^*$.
Then (\ref{eqeq-11}) gives
$e(A,B) \le (|A| -1)|B| -t^*$. Furthermore, (\ref{eqeq-10})
reduces to
\[ \lambda^2 \le |A| + (|A| -1)|B| +t^*
\le  |A| + (|A| -1)|B| +t -1 \le \lfloor \tfrac{n^2}{4}\rfloor, \]
which is a contradiction. Thus, we get $d_B(u_0)=0$.
\end{proof}

In conclusion,  if we assume that
$t^*\le t-1$, then by the items (a), (b) and (c),
we see that $G[A]$ consists of a star $K_{1,t^*}$
with some isolated vertices,
$B$ is an independent set of $G$, and $N_G(w) \subseteq
A \setminus \{u_0\}$
for every vertex $w\in B$, where $u_0$ is the center of the
star $K_{1,t^*}$ in $G[A]$. Consequently,
it follows that $t=t^*$, which contradicts with the
previous assumption.
Therefore, we have $t^*=t$.
\end{proof}

\medskip
{\bf Claim 5.} $G[A] = K_{1,t^*} \cup (|A| - t^* -1)K_1$.

\begin{proof}
Suppose on the contrary that
 $G[A] \neq K_{1,t^*} \cup (|A| - t^* -1)K_1$.
First of all, we can see that $b_1=0$.
 Otherwise, if $b_1\neq 0$,
 then there exists a triangle consisting of
 a vertex of $B_1$ and an edge of $G[A]$,
 which leads to $t^* < t$. This contradicts with Claim 4.
Applying (\ref{eqeq-7}) and Claim 3, we get
$e(A,B) \le (|A| -2) |B|$, equality holds if and only if
$d_A(w)=|A|-2$ for every $w\in B$, that is, $B=B_2$.
By (\ref{eqeq-10}), we get
\begin{align*}
 \lambda^2   \le |A| + 2t^* + (|A| -2) |B|
 \le \begin{cases}
 \frac{n^2}{4} +3 & \text{if $n$ is even;} \\[1mm]
 \frac{n^2}{4} +2 & \text{if $n$ is odd,}
 \end{cases}
 \end{align*}
 where the last inequality holds since $t^*\le \lfloor \frac{n}{2}\rfloor$ and  $|A| + (|A| -2)|B|$ attains the maximum at
 $|B|=\lfloor \frac{n}{2}\rfloor -2$.
 Therefore, if $n$ is odd, then we get a contradiction
 to (\ref{eq-n2+2}).

 In what follows, we shall also show a contradiction
 for even $n$.
 In this case, by Claim 1, we get $|A| \ge \frac{n}{2} +1$
 and $|B| \le \frac{n}{2}-2$.
 If $|B| \neq \frac{n}{2} -2$, then
 $|B| \le \frac{n}{2}-3$ and
 $\lambda^2 \le (\frac{n}{2} +2) + 2\frac{n}{2}
 + \frac{n}{2} (\frac{n}{2} -3) \le \frac{n^2}{4} +2$,
 which is a contradiction.
 Hence we get $|B| = \frac{n}{2} -2$
 and $|A|= \frac{n}{2}+1$. If $t^* \le \frac{n}{2} -1$,
 then we get $\lambda^2 \le
 (\frac{n}{2} +1) + 2(\frac{n}{2} -1)
 + (\frac{n}{2} -1)(\frac{n}{2} -2) = \frac{n^2}{4} +1 $,
 a contradiction.
 Thus, we obtain $t=t^*=\frac{n}{2}$.
 Similarly, if $b_i\neq  0$ for some $i\ge 3$,
 then $e(A,B) \le (|A| -2)|B| -1$, which yields
 $\lambda^2 \le \frac{n^2}{4} +2$, a contradiction.
 Therefore, we have $b_3=b_4=\cdots =b_{|A|}=0$
 and $B=B_2$. To sum up, for even $n$,
 we have proved that
 $|A|=\frac{n}{2} +1$, $|B|=\frac{n}{2}-2$,
 $t^*=t=\frac{n}{2}$ and $d_A(w)=|A|-2$
 for every $w\in B$.

 \medskip
 Next, we continue the proof in two cases, and
 show that $\lambda (G) <\lambda(K_{\frac{n}{2}, \frac{n}{2}}^+) $.

 \medskip
 {\bf Case 1.}  $G[A]$ is a tree.
 Since $G[A]$ is not a star $ K_{1,\frac{n}{2}}$, the longest path of $G[A]$
 has at least $4$ vertices.
 On the other hand, since $t=t^*$ and each vertex of $B$ has exactly
 $|A| -2$ neighbors in $A$,
we can observe that
the longest path of $G[A]$ has at most
 $5$ vertices.   Otherwise,
 if $G[A]$ contains a path $P_6$, say $u_1u_2\cdots u_6$, then for any $w\in B$, there exists an $i$
such that the edge $\{u_i,u_{i+1}\} \subseteq  N(w)$,
since $w$ has exactly two non-neighbors in $A$.
 This leads to $t^* < t$, a contradiction.
Invoking this observation, the longest path in $G[A]$ is either $P_4$ or $P_5$.

\medskip
{\bf Subcase 1.1.}
 If $P_4$ is a longest path in $G[A]$,
 denote by $P_4=u_0u_1u_2u_3$,
 then other vertices of $A$ are adjacent to either
 $u_1$ or $u_2$.
Moreover,  for each vertex $w\in B$, we have
$d_A(w)=|A|-2$. Since $t=t^*$, we get
$N(w)=A \setminus \{u_1,u_2\}$;
see the graph $G_1$ in Figure \ref{fig-P4}.

  \begin{figure}[H]
\centering
\includegraphics[scale=0.8]{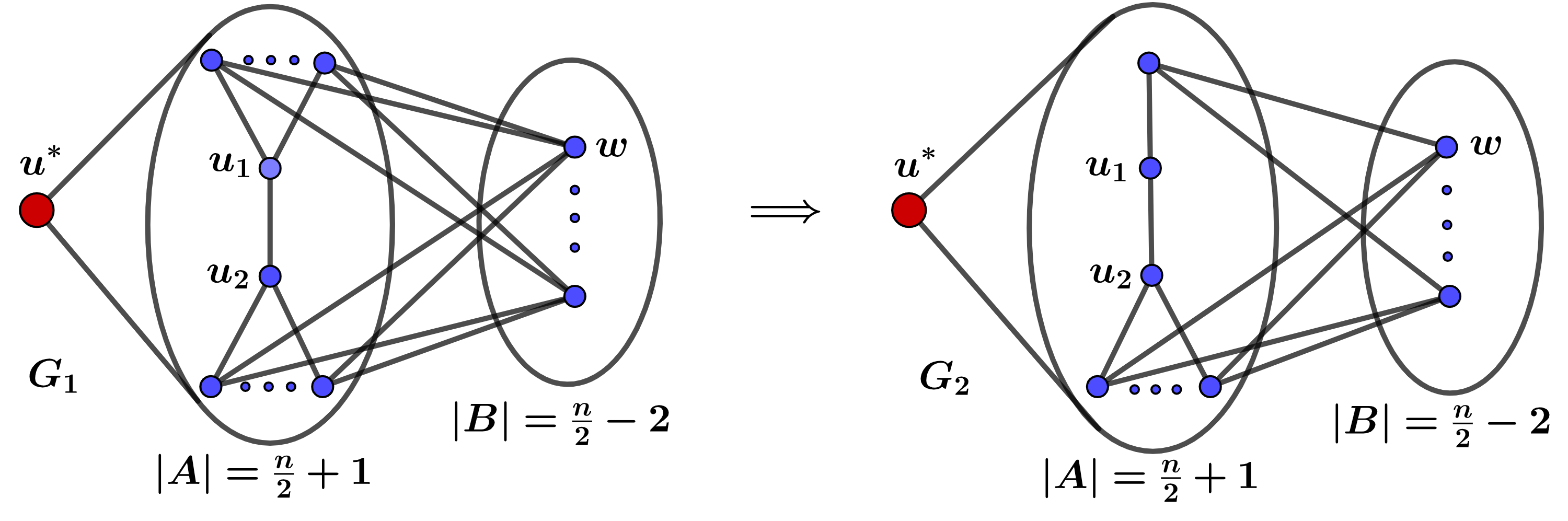}
\caption{The longest path of $G[A]$ is $P_4$.}
\label{fig-P4}
\end{figure}

Without loss of generality, we may assume that
$x_{u_1} \le x_{u_2}$.
We define $G_2$ as the graph obtained from $G_1$
by deleting $d_A(u_1) -1$ neighbors of $u_1$, and joining
these vertices to $u_2$.
By Lemma \ref{lem-22}, we have $\lambda (G_1) < \lambda (G_2)$.
Next, we shall show that $\lambda (G_2) < \lambda(K_{\frac{n}{2}, \frac{n}{2}}^+)$.
By a direct computation,
we know that $\lambda (G_2)$ is the largest root of
\[  g_2(x) := x^6 - (\tfrac{n^2}{4} - \tfrac{n}{2} +3) x^4 -
nx^3 + (\tfrac{5n^2}{4} - 7n +8 ) x^2 + (2n-8)x - (n-4)^2. \]
Recall in Lemma \ref{lem-21} that $\lambda(K_{\frac{n}{2}, \frac{n}{2}}^+)$ is the largest root of
$f(x)=x^3-x^2 - \frac{n^2}{4}x + \frac{n^2}{4}-n $.
For every $x$ satisfying $x^2 \ge \frac{n^2}{4} +2$,
it follows that
\begin{align*}
&g_2(x)- x^3f(x) \\
&=
x^5 + (\tfrac{n}{2} -3) x^4 - \tfrac{n^2}{4}x^3
+ (\tfrac{5}{4}n^2 - 7n +8)x^2 + (2n-8)x - (n-4)^2 \\
&> (\tfrac{n}{2} -3) x^4 + 2x^3
+ (\tfrac{5}{4}n^2 - 7n +8)x^2 + (2x-n+4)(n-4) >0.
\end{align*}
Consequently, we get $g_2(x) > x^3f(x)$ for
every $x \ge \lambda(K_{\frac{n}{2}, \frac{n}{2}}^+)$.
Then $g_2(\lambda(K_{\frac{n}{2}, \frac{n}{2}}^+)) >0$
and $g_2(x) >0$ for every $x\ge
\lambda(K_{\frac{n}{2}, \frac{n}{2}}^+)$,
which implies $\lambda (G_2)<
\lambda(K_{\frac{n}{2}, \frac{n}{2}}^+)$.
This is a contradiction with the assumption.

\medskip
{\bf Subcase 1.2.}
If $P_5$ is a longest path of $G[A]$, say
$P_5=u_0u_1u_2u_3u_4$, then
we have $N(w)= A \setminus \{u_1,u_3\}$
for every $w\in B$. Otherwise,
if $N(w)\neq  A \setminus \{u_1,u_3\}$, then by $d_A(w)=|A| -2$,
there exists a triangle containing $w$, which contradicts with $t^*=t$.
Furthermore, each vertex of $A\setminus \{u_0,\ldots ,u_4\}$
is adjacent to either $u_1$ or $u_3$. Indeed, such a vertex
can not adjacent to $u_0$ or $u_4$ since $P_5$ is a longest path.
Moreover, such a vertex can not be adjacent to $u_2$. Otherwise,
if $uu_2$ is an edge in $G[A]$, then $wuu_2$ is a triangle,
contradicting with $t^*=t$.
Therefore, the structure of $G$ can be determined as below; see
the graph $G_3$ in Figure \ref{fig-P5}.

    \begin{figure}[H]
\centering
\includegraphics[scale=0.8]{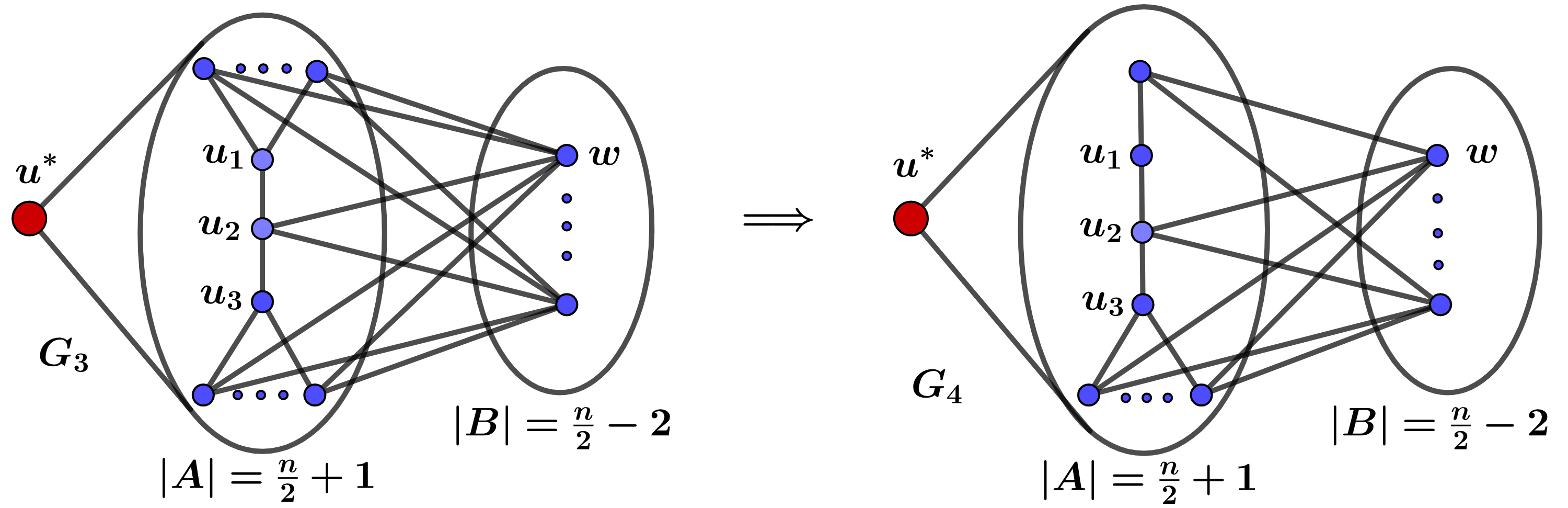}
\caption{The longest path of $G[A]$ is $P_5$.}
\label{fig-P5}
\end{figure}

Similarly, we may assume further that
$x_{u_1} \le x_{u_3}$.
Let $G_4$ be the graph obtained from $G_3$
by removing $d_A(u_1) -1$ neighbors of $u_1$, and adding
these vertices adjacent to $u_3$.
By Lemma \ref{lem-22} again,
we have $\lambda (G_3) < \lambda (G_4)$.
Moreover, we shall show that $\lambda (G_4) < \lambda(K_{\frac{n}{2}, \frac{n}{2}}^+)$.
By computations,
we obtain that $\lambda (G_4)$ is the largest root of
\begin{align*}  g_4(x) &:=
x^7 - (\tfrac{n^2}{4} - \tfrac{n}{2} +3)x^5 - nx^4 +
(\tfrac{5}{4}n^2 - 7n +5) x^3  \\
&\quad   + (3n-16)x^2 - (\tfrac{3}{2}n^2 - 14n +31)x.
\end{align*}
Moreover, for $x^2\ge \frac{n^2}{4} +2$,
one can verify that
\begin{align*}
& \tfrac{1}{x}g_4(x) - x^3 f(x) \\
&= x^5 + (\tfrac{n}{2} -3) x^4 - \tfrac{n^2}{4}x^3
 + (\tfrac{5}{4}n^2 - 7n +5) x^2  + (3n-16) x - (\tfrac{3}{2}n^2 -14n +31) \\
 & > (\tfrac{n}{2} -3)x^4 + 2x^3 +
 (\tfrac{5}{4}n^2 - 7n +5)x^2 +
 (3n-16)x -(\tfrac{3}{2}n^2 -14n +31) >0 .
\end{align*}
In other words, we get $g_4(x) > x^4f(x)$
for every $x\ge  \lambda(K_{\frac{n}{2}, \frac{n}{2}}^+)$.
Recall in Lemma \ref{lem-21} that $ \lambda(K_{\frac{n}{2}, \frac{n}{2}}^+)$ is the largest root of $f(x)$.
As a result, it follows that
$g_4( \lambda(K_{\frac{n}{2}, \frac{n}{2}}^+)) > 0$, and
$g_4(x) >0$ for every $x\ge  \lambda(K_{\frac{n}{2}, \frac{n}{2}}^+)$, which yields $\lambda (G_4) <
 \lambda(K_{\frac{n}{2}, \frac{n}{2}}^+)$.

 \medskip
  {\bf Case 2.}  $G[A]$ is not a tree.
  In this case, $G[A]$ has a cycle.
  We claim that $G[A]$ has no cycle of  length at least $5$.
  Otherwise, if $G[A]$ contains a cycle $C_t$ with $t\ge 5$, then there exists a triangle consisting of
  an edge of $C_t$ and a vertex of $B$, since
  each vertex of $B$ has exactly two non-neighbors in
  $A$. Therefore, we get $t^*<t$, which is a contradiction.
  Similarly, we can see that
  $G[A]$ has no triangle.
  Thus, the longest cycle of $G[A]$ is $C_4$,
  and the copy of $C_4$ is an induced copy.
  Otherwise, if this copy of $C_4$ is not induced,
  then it has a chord, and so
  $G[A]$ contains a triangle, which contradicts with
  $t^*=t$.
 Assume that $C_4=u_0u_1u_2u_3$ is the longest cycle of $G[A]$.

 \medskip
 {\bf Subcase 2.1.} $|A|=5$, that is, $n=8$ and $|B|=2$.
 Then $G[A]=C_4 \cup K_1$.
 For each $w\in B$, we know that $w$ has exactly two
 non-neighbors in $\{u_0,u_1,u_2,u_3\}$.
 Since $t^*=t$ and there is no triangle containing any vertex
 $w\in B$,
we have two possibilities, i.e., either $N_A(w)=
 A \setminus \{u_1,u_2\}$ or $N_A(w)= A \setminus \{u_0,u_3\}$.
 Thus, $G$ is determined as  $G_5$ or $G_6$
 in Figure \ref{fig-C4=n8}.
 By calculations, we obtain that $\lambda (G_5)\approx 3.934 <  \lambda(K_{4, 4}^+)$,
 and $\lambda (G_6) \approx 3.884 < \lambda(K_{4, 4}^+)$,
 which  contradicts with the assumption.

         \begin{figure}[H]
\centering
\includegraphics[scale=0.8]{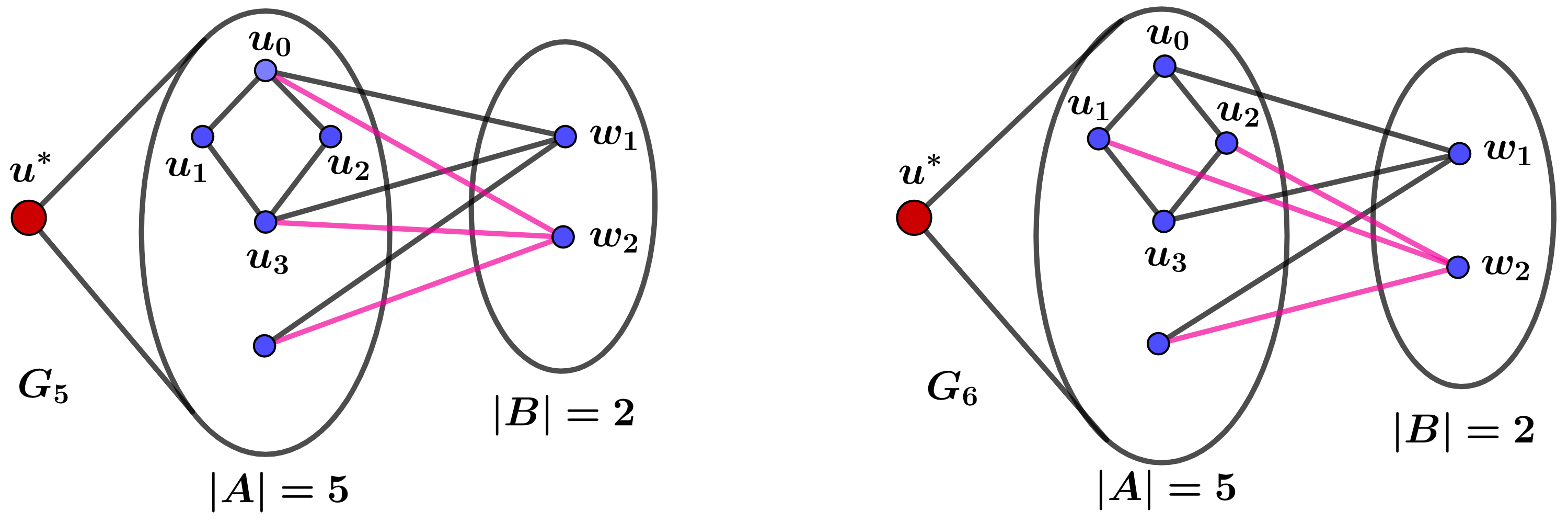}
\caption{The longest cycle of $G[A]$ is $C_4$ and $|A|=5$.}
\label{fig-C4=n8}
\end{figure}

 {\bf Subcase 2.2.} $|A|\ge 6$.
 In this case, we have $n\ge 10$ and $e(A)=|A| -1 \ge 5$.
 Keeping in mind that $t^*=t$, and
 there is no triangle containing any vertex $w\in B$.
Moreover, each vertex $w \in B$ has exactly two non-neighbors in $ \{u_0,u_1,u_2,u_3\}$.
Fixing a vertex $w_0\in B$, we may assume that
 $N_A(w_0)= A \setminus \{u_1,u_2\}$.
 For each $u\in A \setminus \{u_0,u_1,u_2,u_3\}$,
we have $uw \in E(G)$, and $u$ can not be adjacent to
$u_0$ or $u_3$, that is, $N_A(u) \subseteq \{u_1,u_2\}$.
Since $e(A)\ge 5$,
there exists an edge in $G[A]$ outside $C_4$,
say $uu_2 \in E(G)$, which leads to
$N_A(w) = A \setminus \{u_1,u_2\}$ for other vertex
$w\in B\setminus \{w_0\}$.
Otherwise, if there is a vertex $w\in B\setminus \{w_0\}$ such that
$N_A(w) \neq A \setminus \{u_1,u_2\}$, then
$N_A(w) = A \setminus \{u_0,u_3\}$, i.e.,
$w$ is adjacent to both $u_1$ and $u_2$,  so $wu_2u$
forms a triangle, contradicting with $t^*=t$.
Thus, the structure of $G$ is determined as in Figure \ref{fig-C4}.

        \begin{figure}[H]
\centering
\includegraphics[scale=0.8]{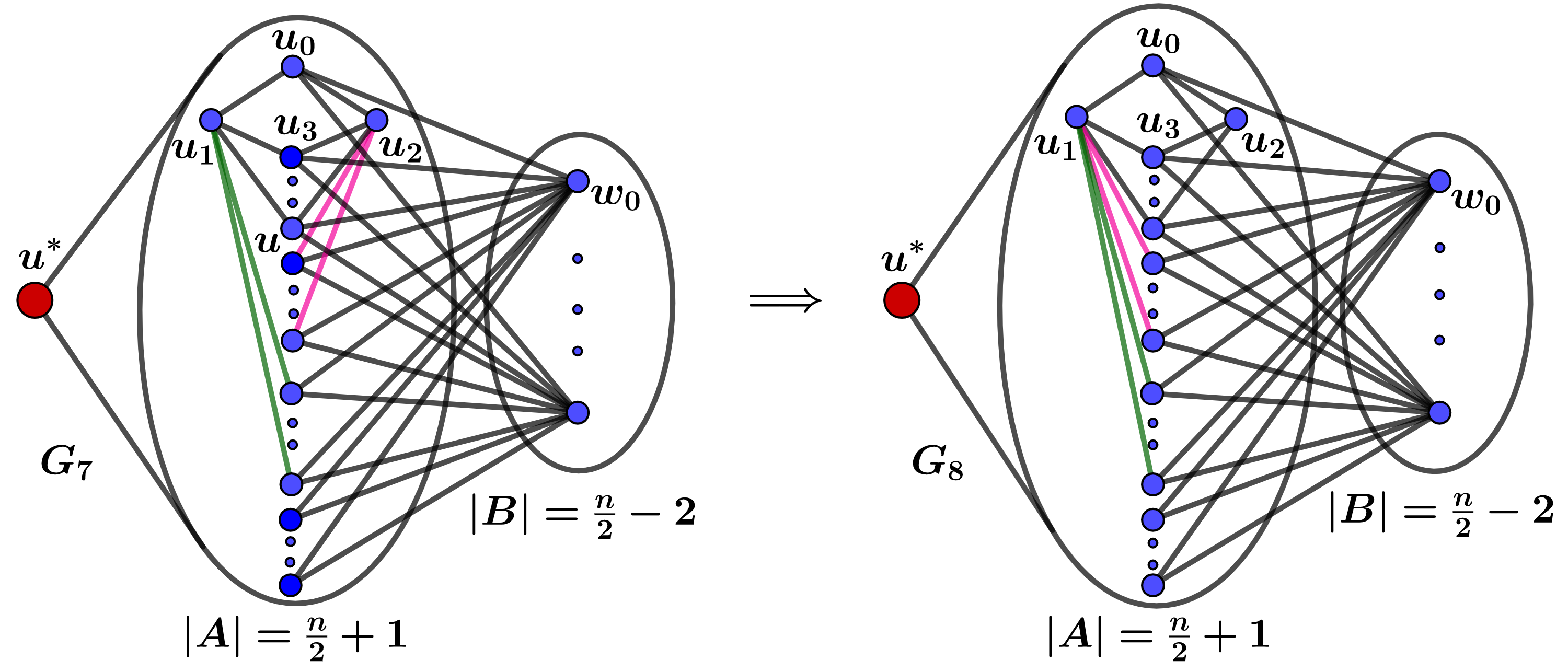}
\caption{The longest cycle of $G[A]$ is $C_4$
and $|A| \ge 6$.}
\label{fig-C4}
\end{figure}

Without loss of generality,
we assume further that $x_{u_1} \ge x_{u_2}$.
By Lemma \ref{lem-22}, removing the vertices that are adjacent only to $u_2$, and adding them adjacent to $u_1$,
we get a new graph $G_8$ with larger spectral radius
than $G_7$;
see Figure \ref{fig-C4}. Finally,
some tedious computations can yield  $\lambda (G_8) <
 \lambda(K_{\frac{n}{2}, \frac{n}{2}}^+)$.
Indeed, let $a=|N_A(u_1) \cap N_A(u_2)|$
be the number of vertices of $A$ adjacent to both $u_1$ and $u_2$,
let $b=|N_A(u_1)\setminus N_A(u_2)|$ be the
number of vertices of $A$ adjacent only to $u_1$,
and let $c$ be the number of  isolated vertices of $G[A]$.
Then $a+b+c=\frac{n}{2}-1$ and $2a+b=e(A)=\frac{n}{2}$.
Upon computation, we obtain that $\lambda (G_8)$
is the largest root of
\begin{equation*}
\begin{aligned}
g_8(x)&:=x^7- (\tfrac{a+b+c}{2}n + a+2-c)x^5 -
(4a+2b)x^4 \\
&\quad + ((a+b+c)n+ \tfrac{ab}{2}n +ac n + \tfrac{bc}{2}n -3b-4c -2ac -bc-4a  )x^3 \\
&\quad + 2abx^2 + (2ab+2bc +abc - \tfrac{ab}{2}n
-\tfrac{bc}{2}n - \tfrac{abc}{2}n)x.
\end{aligned}
\end{equation*}
It is easy to calculate that
for every $x^2 > \frac{n^2}{4} +2$,
\begin{align*}
&\tfrac{1}{x}g_8(x) - x^3f(x) \\
&= x^5 + (\tfrac{n^2}{4} - \tfrac{a+b+c}{2}x + c-a-2)x^4
- (\tfrac{n^2}{4} -n +4a +2b) x^3 \\
& \quad + ((a+b+c)n + \tfrac{ab}{2}n +ac n + \tfrac{bc}{2}n
-3b -4c - 2ac - bc -4a)x^2 \\
&\quad + 2abx + (2ab+2bc +abc - \tfrac{ab}{2}n
-\tfrac{bc}{2}n - \tfrac{abc}{2}n)  \\
& >  (\tfrac{n^2}{4} - \tfrac{a+b+c}{2}x + c-a-2)x^4
+ (n+2-4a -2b) x^3 \\
&\quad + ((a+b+c)n + \tfrac{ab}{2}n +ac n + \tfrac{bc}{2}n
-3b -4c - 2ac - bc -4a)x^2 \\
&\quad + 2abx + (2ab+2bc +abc - \tfrac{ab}{2}n
-\tfrac{bc}{2}n - \tfrac{abc}{2}n) >0.
\end{align*}
Henceforth, we get $g_8(x) > x^4f(x)$
for every $x\ge  \lambda(K_{\frac{n}{2}, \frac{n}{2}}^+)$.
Furthermore, it follows that $g_8( \lambda(K_{\frac{n}{2}, \frac{n}{2}}^+)) >0$ and $g_8(x) >0$ for every $x\ge  \lambda(K_{\frac{n}{2}, \frac{n}{2}}^+)$, which implies $\lambda (G_8) <
 \lambda(K_{\frac{n}{2}, \frac{n}{2}}^+)$, a contradiction.
\end{proof}

\medskip
{\bf Claim 6.} Let $u_0$ be the central vertex of the star $K_{1,t^*}$ in $G[A]$.
Then $d_B(u_0) =0$.

\begin{proof}
The proof is similar with that of Part (c) in Claim 4.
Assume on the contrary that
there exists an integer
$i\in \{1,2,\ldots ,|A|\}$ and $w_i \in B_i$ such that
$u_0w_i$ is an edge in $G[B]$.
Since $d_A(w_i) = |A| -i$ and $b_0=0$,
we obtain by (\ref{eqeq-7}) that
 \begin{equation*}
 e(A,B) \le (|A| -1)|B| - (i-1).
 \end{equation*}

If $i\le t^*-1$, then by Claim 5,
there are at least $t^*-i \ge 1$ triangles
consisting of $u_0,w_i$ and a leaf vertex of the star $K_{1,t^*}$,
which leads to  $t>t^*$, contradicting with Claim 4.
Hence, we get $ i\ge t^*$.
Then
\[  e(A,B) \le (|A| -1)|B| -t^* +1. \]
 Furthermore, inequality (\ref{eqeq-10})
reduces to
\[ \lambda^2 \le |A| + (|A| -1)|B| +t^* +1
\le  |A| + (|A| -1)|B| +t +1
\le \lfloor \tfrac{n^2}{4}\rfloor +2, \]
which is a contradiction. Thus, we get $d_B(u_0)=0$.
\end{proof}

\medskip
{\bf Claim 7.} $e(B)=0$, that is, $B$ is an independent set in $G$.

\begin{proof}
Suppose on the contrary that
$\{w_i,w_j\}$ is an edge of $G[B]$,
where $w_i\in B_i$ and $w_j \in B_j$.
By Claim 4, we get $t=t^*$, which implies that
$N_A(w_i) \cap N_A(w_j) = \varnothing$.
By Claim 6, we know that $N_A(w) \subseteq
A \setminus \{u_0\}$ for every $w\in B$.
It follows that $(|A|-i) + (|A|-j) \le |A| -1$, which yields
$i+j \ge |A| +1$.
Recall in Claim 3 that $b_0=0$. Then
\[  e(A,B) \le |A||B| - (|B| -2) -i-j \le
|A||B| - |B| -|A| +1. \]
By (\ref{eqeq-10}), we get
\[  \lambda^2 \le |A| + 2t^* + e(A,B)
\le |A||B| - |B| +2t^* +1. \]
Since $|A||B| - |B|$ attains the maximum at
$|B| = \lfloor \frac{n}{2}\rfloor -1$
and $t^* \le \lfloor \frac{n}{2}\rfloor$, we get
$\lambda^2 \le \lceil \frac{n}{2}\rceil (\lfloor \frac{n}{2}\rfloor -1) - (\lfloor \frac{n}{2}\rfloor -1) + 2
\lfloor \frac{n}{2}\rfloor +1\le
\lfloor \frac{n^2}{4} \rfloor +2$, a contradiction.
\end{proof}

Finally,
we are ready to show that $G $ is the desired extremal graph $K_{\lceil \frac{n}{2} \rceil, \lfloor \frac{n}{2} \rfloor}^+$.

\medskip
{\bf Claim 8.}  $N_G(w)=A\setminus \{u_0\}$ for each $w\in B$.

\begin{proof}
By  Claims 4, 5, 6 and 7 above,
we know that all triangles in $G$ must contain
the vertex $u^*$, that is, $t^*=t\le \lfloor \frac{n}{2}\rfloor$.
Moreover, $G[A]$ consists of a star $K_{1,t^*}$ on the center $u_0$ together with $|A| -t^*-1$ isolated vertices.
Furthermore, we have $N_A(w)\subseteq A \setminus \{u_0\}$ for every $w\in B$.
In addition, we observe that both $A\setminus \{u_0\}$
and $B$ are independent sets, and
adding edges between $B$ and $A\setminus \{u_0\}$
does not increase the number of triangles.
In our assumption, $G$ is chosen as a graph with the maximum spectral radius,
Thus we obtain $N_G(w)=A\setminus \{u_0\}$ for each $w\in B$.
\end{proof}

Let $A_0:=\{u\in A: d_A(u)=0\}$ be the set of
isolated vertices in $G[A]$,
and denote by $A_1=A \setminus A_0$.
Then, we have $|A_1|=t+1$ and $G[A_1]=K_{1,t}$.

\medskip
{\bf Claim 9.} We obtain that $\sum_{u\in A_0} x_u < x_{u^*}$.

\begin{proof}
Otherwise,  if $\sum_{u\in A_0}  x_u \ge x_{u^*}$,
then we define a new graph $G'$ on vertex set $V(G')=V(G)$
with edge set
\[  E(G')=E(G) - \{u_0u^*\} +
\{u_0u : u\in A_0\}. \]
Observe that $G'$ is a bipartite graph  on
$n$ vertices with two color classes $
A\setminus \{u_0\}$ and $B \cup \{u^*,u_0\}$.
Then Nikiforov's result  implies
$\lambda (G') \le \lambda(T_{n,2})
= \sqrt{\lfloor {n^2}/{4}\rfloor}$.
On the other hand, we can show that
$  \lambda (G) \le \lambda (G')$.  Indeed, since $\bm{x}$ is unit, Rayleigh's formula gives
$\lambda(G') \ge \bm{x}^TA(G')\bm{x}$ and then
\[ \lambda (G') - \lambda (G)
\ge \bm{x}^TA(G')\bm{x} - \bm{x}^TA(G)\bm{x} =
2x_0 \left(\sum_{u\in A_0} x_u -x_{u^*} \right) \ge 0, \]
which yields $\lambda (G) \le \lambda (G')\le \sqrt{\lfloor {n^2}/{4}\rfloor}$,
a contradiction\footnote{In fact, it follows further that $\lambda(G) < \lambda (G')$
since $\bm{x}$ can not be a Perron vector of $G'$.}.
\end{proof}

\medskip
{\bf Claim 10.} $|A|=t+1$, that is, $A=A_1$ and $A_0=\varnothing$.

\begin{proof}
Note that Claim 8 gives $N_G(u)= \{u^*\} \cup B$ for every $u\in A_0$,
which implies $\lambda x_u = x_{u^*} + \sum_{w\in B} x_w$.
By Claim 9, it follows  that
\[   \lambda x_{u^*} > \sum_{u\in A_0} \lambda x_u = |A_0| x_{u^*} + |A_0| \sum_{w\in B} x_w.\]
Thus, by Claim 8 again, we obtain
\[  \sum_{w\in B} d_A(w) x_w
= (|A| -1) \sum_{w\in B} x_w < (|A| -1) \frac{\lambda - |A_0|}{|A_0|} x_{u^*} \le
 \frac{(|A| -1)(t+1)}{|A|-t-1} x_{u^*},  \]
 where the last inequality holds by
 $\lambda \le |A|$ and $|A_0| = |A| -t-1$.
 Combining this
 with (\ref{eqeq-6}) and (\ref{eq-2eA}),
 we have
 \begin{equation} \label{eqeq-haa}
   \lambda^2 x_{u^*}
 \le |A|x_{u^*} +2tx_{u^*} +
 \sum_{w\in B} d_A(w)x_w
< \left( |A| +2t+  \frac{(|A| -1)(t+1)}{|A|-t-1}  \right) x_{u^*} .
\end{equation}
For convenience, we denote
\[ h(|A|) := |A| +2t + \frac{(|A| -1)(t+1)}{|A|-t-1}. \]
Clearly, we have $|A| \ge |A_1|= t+1$.
Next, we proceed the proof in three cases.

\medskip
{\bf Case 1.}
$|A| \ge t+4$.

In this case,
to obtain a contradiction, we shall show that
$\lambda^2 \le \lfloor \frac{n^2}{4} \rfloor +2$.
Since $|A| \in [t+4, n-2]$ and $h(|A|)$ is convex in terms of  $|A|$,
we get from (\ref{eqeq-haa}) that
\begin{equation} \label{ha}
\lambda^2< h(|A|) \le \max\{h(t+4), h(n-2)\}.
 \end{equation}
To determine the maximum in (\ref{ha}),
we divide the proof in two subcases.

{\bf Subcases 1.1.} $t\le \lfloor \frac{n}{2}\rfloor -1$.
By computation, we get
\[ h(t+4) = \frac{1}{3}t^2 + \frac{13}{3}t+5
\le \frac{1}{3} \left(\lfloor \frac{n}{2}\rfloor -1 \right)^2
+ \frac{13}{3}\left(\lfloor \frac{n}{2}\rfloor -1 \right) +5.   \]
One can check that $h(t+4) \le \lfloor \frac{n^2}{4}\rfloor +2$ for every even $n\ge 12$ and odd $n\ge 9$.
On the other hand, we have
\[ h(n-2)
=(n-2) +2t + \frac{(n-3)(t+1)}{n-t-3}
 \le n-2 + 2 \lfloor \frac{n}{2}\rfloor -2
+ \frac{(n-3)\lfloor \frac{n}{2} \rfloor}{ n -\lfloor \frac{n}{2}\rfloor -2},  \]
where the inequality holds due to $t\le \lfloor \frac{n}{2}\rfloor -1$.
Then we can  verify that
$h(n-2) \le \lfloor \frac{n^2}{4}\rfloor +2$ for every $n$.
To sum up, for $n=9$ and $n\ge 11$,
we get from (\ref{ha}) that $\lambda^2 \le h(|A|)
\le \lfloor \frac{n^2}{4}\rfloor +2$, a contradiction.
Next, we consider the remaining cases $n\le 8$ and $n=10$.
By Claim 2, we have $t\ge 1$ and $ |B|\ge 1$.
Recall that $|A| \ge t+4 \ge 5$ and $|A| \le n-2$.
Thus, we get $n\in \{7,8,10\}$.
For $n=7$, we must have $t=1$, $|A|=5$ and $|B|=1$.
Since Claim 4 gives $t^*=t$,
Claim 8 implies $e(A,B)= (|A|-1)|B|$,
then by (\ref{eqeq-10}), we get
\begin{equation} \label{eqeq-10-f}
\lambda^2 \le |A| +2t + (|A| -1)|B|.
\end{equation}
Consequently, for $n=7$, we obtain $\lambda^2 \le 11< \lfloor \frac{n^2}{4}\rfloor +2$, a contradiction.
For $n=8$, we have $(t,|A|)=(2,6),(1,5)$ or $(1,6)$.
For $n=10$, we have
$(t,|A|)=(4,8),(3,7),(3,8)$, $(2,6),(2,7),(2,8)$,
$(1,5),(1,6),(1,7),(1,8)$. Similarly, using (\ref{eqeq-10-f}), we can verify that $\lambda^2 < \lfloor \frac{n^2}{4}\rfloor +2$. This is a contradiction.

{\bf Subcases 1.2.} $t= \lfloor \frac{n}{2}\rfloor$.
In this case, we have
\[ h(t+4)\le \frac{1}{3} \lfloor \frac{n}{2}\rfloor ^2
+ \frac{13}{3} \lfloor \frac{n}{2}\rfloor  +5  \]
 It is easy to verify that $h(t+4)
 \le \lfloor \frac{n^2}{4}\rfloor +2$ holds for even $n\ge 16$
 and odd $n\ge 13$.

  Moreover, we have
  \[  h(n-2) \le n-2 + 2 \lfloor \frac{n}{2}\rfloor
+ \frac{(n-3)(\lfloor \frac{n}{2} \rfloor +1)}{ n -\lfloor \frac{n}{2}\rfloor -3}.\]
One can check that
 $h(n-2) \le \lfloor \frac{n^2}{4}\rfloor +2$ for every $n\ge 13$.
 Therefore, for odd $n\ge 13$ and even $n\ge 16$,
 we get $\lambda^2 \le h(|A|) \le  \lfloor \frac{n^2}{4}\rfloor +2$, a contradiction.
In addition, we have $n-2\ge |A| \ge t+4=\lfloor \frac{n}{2}\rfloor +4$, which
gives rise to $n\ge 11$, and then $n\in \{11,12,14\}$.
 For $n=11$, we have $t=5$, $|A|=9$ and $|B|=1$.
 By (\ref{eqeq-10-f}), we get $\lambda^2 \le 27 < \lfloor \frac{n^2}{4}\rfloor +2 $, a contradiction.
For $n=12$, we have $t=6$, $|A|=10$ and $|B|=1$,
which leads to $\lambda^2 \le 31 <  \lfloor \frac{n^2}{4}\rfloor +2 $, a contradiction.
For $n=14$, we have $t=7$,  and $(|A|,|B|)=(11,2)$ or
$(12,1)$, which yields anyway that
$\lambda^2 <  \lfloor \frac{n^2}{4}\rfloor +2 $, a contradiction.

\medskip
{\bf Case 2.}
$|A| = t+3$.

By (\ref{eqeq-haa}), we need to calculate the value $h(|A|)$.
If $t\le \lfloor \frac{n}{2}\rfloor -1$,
then
\[  h(|A|)=h(t+3)= \frac{1}{2}t^2 + \frac{9}{2}t +4
\le \frac{1}{2}\left(\lfloor \frac{n}{2}\rfloor -1 \right)^2 +
\frac{9}{2}\left(\lfloor \frac{n}{2}\rfloor -1 \right) +4. \]
For even $n\ge 14$
and odd $n\ge 9$, one can verify that
$h(t+3)\le \lfloor \frac{n^2}{4} \rfloor +2$, a contradiction.
Since $4\le |A|=t+3 \le n-2$, we have $n\ge 6$. Thus,
we need to check the remaining case $n\in \{6,7,8,10,12\}$.
For $n=6$, we have $t=1$ and $|A|=4$,
it follows from (\ref{eqeq-10-f}) that $\lambda^2 \le
9< \lfloor \frac{n^2}{4} \rfloor +2$.
For $n=7$, we have $(t,|A|)=(1,4)$ or $(2,5)$.
Using (\ref{eqeq-10-f}), we get
$ \lambda^2 < 13 < \lfloor \frac{n^2}{4} \rfloor +2$, a contradiction.
For $n=8$, we have $(t,|A|)=(1,4),(2,5)$ or $(3,6)$, in which one can verify by (\ref{eqeq-10-f}) that
$ \lambda^2 \le \lfloor \frac{n^2}{4} \rfloor +2$.
Moreover, the similar argument also holds for $n=10$ and $n=12$, respectively. This is a contradiction.

If $t=\lfloor \frac{n}{2}\rfloor$, then
$h(|A|)=h(t+3)= \frac{1}{2}\lfloor \frac{n}{2}\rfloor^2 +
\frac{9}{2} \lfloor \frac{n}{2}\rfloor +4$.
By computation, we know that
$h(|A|) \le \lfloor \frac{n^2}{4}\rfloor +2$
for even $n\ge 20$ and odd $n\ge 17$,
which leads to $\lambda^2 \le \lfloor \frac{n^2}{4} \rfloor +2$, a contradiction.
Next, we will check the exceptions for $n\le 16$ and $n=18$.
Since $n-2 \ge |A| = t+3=\lfloor \frac{n}{2}\rfloor +3$,
we have $n\ge 9$, and the following cases in Table \ref{tab-A=t3}.

\begin{table}[H]
\centering
\begin{tabular}{ccccccccccc}
\toprule
 $n$ & 9 &  10 & 11 & 12 & 13 & 14 & 15 & 16  & 18 \\
\midrule
 $t$ & 4 & 5 & 5 & 6 &  6 & 7 & 7 & 8  & 9 \\
 $|A|$ & 7  & 8 & 8 & 9 & 9 & 10 & 10 &11 &  12 \\
 $|B|$ & 1 & 1 & 2 & 2 & 3 & 3 & 4 & 4 &  5 \\
\bottomrule
\end{tabular}
\caption{Some examinations in the case
$|A|=t+3$ and $t=\lfloor \frac{n}{2}\rfloor$.}
 \label{tab-A=t3}
\end{table}
For $n\in \{9,10,11,12,14\}$,
one can directly verify by (\ref{eqeq-10-f})
that $\lambda^2 \le \lfloor \frac{n^2}{4} \rfloor +2$,
which is a contradiction.
For the remaining cases $n\in \{13,15,16,18\}$,
to deduce a contradiction, we need to refine the bound in  (\ref{eqeq-10-f}).
Since $G[A] = K_{1,t}\cup 2K_1$ and $u_0$ is the center, we have
\begin{equation} \label{eqeqeq-1}
  \sum_{u\in A}d_A(u)x_u = tx_{u_0} + \sum_{u\in N_A(u_0)} x_u
\le  tx_{u_0} + tx_{u^*} .
\end{equation}
Recall that
$N_G(w) = A \setminus \{u_0\}$
for every $w\in B$.
We get $ \lambda x_w = \sum_{u\in A\setminus \{u_0\}}x_u
=\lambda x_{u^*} - x_{u_0}$ and then $x_w= x_{u^*} - \frac{x_{u_0}}{\lambda}$.
Therefore, we obtain
\begin{equation} \label{eqeqeq-2}
  \sum_{w\in B} d_A(w)x_w= (|A| -1)|B|x_w
= (|A| -1)|B| \left(x_{u^*} - \frac{x_{u_0}}{\lambda} \right).
\end{equation}
Then by (\ref{eqeq-6}), (\ref{eqeqeq-1}) and (\ref{eqeqeq-2}), we get
\[ \lambda^2 x_{u^*} \le |A| x_{u^*}
 + tx_{u^*} +(|A|-1)|B| x_{u^*} +
 \left(t- \frac{(|A| -1)|B|}{\lambda} \right) x_{u_0}.  \]
Recall that $\lambda \le |A|$. It follows that
\begin{equation} \label{eqeq-refine}
  \lambda^2 \le |A| + 2t + (|A| -1) |B|
- \frac{(|A| -1)|B|}{|A|}.
\end{equation}
Clearly, the upper bound in (\ref{eqeq-refine}) refines that in (\ref{eqeq-10-f}).
In particular, for $n=13$,
we get from Table \ref{tab-A=t3} that $t=6$, $|A|=9$ and $|B|=3$. Then
(\ref{eqeq-refine}) yields
$\lambda^2 \le 45- \frac{24}{9}  <
\lfloor \frac{n^2}{4}\rfloor +2$.
For $n=15$, by (\ref{eqeq-refine}), we have
 $\lambda^2 \le 60 - 3.6 <
\lfloor \frac{n^2}{4}\rfloor +2 $.
Similarly, for $n=16$ and $n=18$,
 some computations also lead to $\lambda^2 <
\lfloor \frac{n^2}{4}\rfloor +2 $.
This is a contradiction.

\medskip
{\bf Case 3.} $|A| = t+2$.
We present the proof in two subcases.

{\bf Subcase 3.1.} $n$ is even.
Then $t\le \frac{n}{2}$.
Recall in Claim 1 that $|A| \ge \frac{n}{2} +1$.
Hence, we have two possibilities, namely,
$(|A|,t)= (\frac{n}{2} +1, \frac{n}{2}-1)$
or $(\frac{n}{2}+2, \frac{n}{2})$.
In the former case, we have $|B|=n-1-|A|=\frac{n}{2}-2$.
Using (\ref{eqeq-refine}), we obtain
\[ \lambda^2 \le \frac{n^2}{4} + \frac{n}{2} -1
- \frac{\frac{n}{2}(\frac{n}{2} -2)}{\frac{n}{2}+1}
< \frac{n^2}{4} +2.  \]
For the latter case, we have $|B|=\frac{n}{2}-3$.
We remark  that it is not sufficient to apply the bound (\ref{eqeq-refine}) only. In this case, the structure of $G$ is uniquely determined.
By a direct calculation,
we know that $ \lambda (G)$ is the largest root of
\[ p_1(x) :=x^5 + (1-\tfrac{n^2}{4} )x^3 - nx^2
+ \tfrac{n^2}{2}x -2nx -3x.  \]
By Lemma \ref{lem-21},
$\lambda (K_{\frac{n}{2}, \frac{n}{2}}^+)$
is the largest root of
$f(x)= x^3 -x^2 - \frac{n^2}{4} x+ \frac{n^2}{4} -n$.
Then
\begin{align*}
p_1(x) - x^2 f(x) =
x^4 +x^3 -\tfrac{n^2}{4}x^2 +\tfrac{n^2}{2}x -2nx -3x .
\end{align*}
One can verify that
$p_1(x) > x^2f(x)$ for every $x\ge \frac{n}{2}$.
Since $\lambda (K_{\frac{n}{2}, \frac{n}{2}}^+)
> \lambda (K_{\frac{n}{2}, \frac{n}{2}})= \frac{n}{2}$, it follows that $p_1(\lambda (K_{\frac{n}{2}, \frac{n}{2}}^+)) >0$. Moreover,
$p_1(x) >0$ for every $x\ge \lambda (K_{\frac{n}{2}, \frac{n}{2}}^+)$.
Thus, we obtain $\lambda (G) < \lambda (K_{\frac{n}{2}, \frac{n}{2}}^+)$. This contradicts with
the assumption.

\medskip
{\bf Subcase 3.2.} $n$ is odd.
In this subcase,
Claim 1 implies $|A| \ge \frac{n+1}{2}$
and $t\le \frac{n-1}{2}$.
Thus, there are two possibilities, more precisely,
$(|A|, t)=(\frac{n+1}{2}, \frac{n-1}{2} -1)$ or
$(\frac{n+1}{2}+1, \frac{n-1}{2})$.
In the first case,
we have $|B|=\frac{n-1}{2} -1$.
Applying (\ref{eqeq-refine}),
we get
\[ \lambda^2 \le
\frac{n^2-1}{4} +\frac{n-3}{2} -
 \frac{(\frac{n+1}{2}-1)(\frac{n-1}{2}-1)}{\frac{n+1}{2}}
< \frac{n^2-1}{4}  +1.  \]
For the second case,
we have $|B|=\frac{n-1}{2}-2$.
Nevertheless, it is also not sufficient to use (\ref{eqeq-refine}).
By computation, we obtain that
$\lambda (G)$ is the largest root of
\[  p_2(x) := x^5 + \tfrac{1-n^2}{4}x^3 + (1-n)x^2
 +\tfrac{n^2}{2}x - 2nx - \tfrac{1}{2}x. \]
In view of Lemma \ref{lem-21}, we know that
$\lambda(K_{\frac{n+1}{2}, \frac{n-1}{2}}^+)$
is the largest root of
$ g(x)=x^3 - x^2 + \frac{1-n^2}{4} x +
\frac{n^2}{4} - n + \frac{3}{4}$.
Upon computation, we get
\[  p_2(x)- x^2g(x) =
x^4 +\tfrac{1}{4}x^2 - \tfrac{n^2}{4}x^2 +
\tfrac{n^2}{2}x -2nx -\tfrac{x}{2}. \]
It is easy to check that
$p_2 (x)> x^2g(x)$ for every $x\ge \frac{n-1}{2}$.
Observe that $\lambda(K_{\frac{n+1}{2}, \frac{n-1}{2}}^+)
> \lambda (T_2(n))> \frac{n-1}{2}$.
It follows that $p_2(\lambda(K_{\frac{n+1}{2}, \frac{n-1}{2}}^+)) >0$ and moreover
$p_2(x) >0$ for every $x\ge \lambda(K_{\frac{n+1}{2}, \frac{n-1}{2}}^+)$.
Thus, we get $\lambda (G) < \lambda(K_{\frac{n+1}{2}, \frac{n-1}{2}}^+)$, contradicting with the assumption.
\end{proof}

\medskip
From the  above discussions,
we conclude that $|A| = t+1$.
If $n$ is even, then Claim 1 gives $|A| \ge \frac{n}{2}+1$.
Note that $|A| = t+1 \le \frac{n}{2}+1$.
Thus, we must have $|A| = \frac{n}{2} +1, t=\frac{n}{2}$
and $|B|=\frac{n}{2} -2$. Putting $\{u^*,u_0\}$
and $B$ together, one can observe that
$G=K_{\frac{n}{2},\frac{n}{2}}^+$,
which is the desired extremal graph.
If $n$ is odd, then Claim 1 implies
$|A| \ge \frac{n+1}{2}$.
Moreover, we have $t\le \frac{n-1}{2}$ and $|A| =t+1 \le
\frac{n+1}{2}$. Therefore, we obtain
$|A| = \frac{n+1}{2}, t=\frac{n-1}{2}$ and
$|B|=\frac{n-1}{2} -1$.
Moving $\{u^*,u_0\}$ and $B$ together,
we see that $G=K_{\frac{n+1}{2}, \frac{n-1}{2}}^+$,
as desired.
This completes the proof.
\end{proof}

\section{Concluding remarks}

\label{sec5}

In the Concluding remarks of \cite{NZ2021}, Ning and Zhai wrote that
for the case of triangles, is there some interesting
phenomenon when we consider the relationship between the number of triangles
and signless Laplacian spectral radius, Laplacian spectral radius, distance spectral
radius and etc?

\medskip 
We answer here that
the counting result of triangles does {\sc not} hold for the signless Laplacian spectral radius.
A result of He, Jin and Zhang \cite{HJZ2013} implies that
if $G$ is a triangle-free graph on $n$ vertices,
then the signless Laplacian radius $q(G) \le n$,
with equality  if and only if $G$
is a complete bipartite graph.
However, for an $n$-vertex graph $G$ with $q(G) >n$,
it is possible that $G$ has exactly one triangle.
Indeed,  adding one edge into the independent set of
$K_{n-1,1}$, we get a new graph, denote $K_{n-1,1}^+$, which  satisfies $q(K_{n-1,1}^+)
> n$, but $t(K_{n-1,1}^+)=1$.
Although the counting result  in terms of $q(G)$
does not hold for triangles,
it seems possible that the spectral counting result
is feasible for cliques $K_{r+1}$ for every  $r\ge 3$,
since it was proved in \cite{HJZ2013}  that for $r\ge 3$,
the $r$-partite Tur\'{a}n graph $T_{n,r}$ is the unique $K_{r+1}$-free graph attaining the maximum of
the signless Laplacian spectral radius.

\medskip 
A well-known result of Nikiforov \cite{Niki2006}
implies that if $G$ is a triangle-free graph with
$m$ edges, then $\lambda (G) \le \sqrt{m}$,
and equality holds if and only if $G$ is a complete bipartite
graph (possibly with some isolated vertices).
Furthermore, Ning and Zhai \cite{NZ2021}
proved a counting result, which asserts that if
$  \lambda (G)\ge \sqrt{m}$,
then $G$ has at least $\lfloor \frac{\sqrt{m}-1}{2} \rfloor$  triangles, unless $G$ is a  complete bipartite graph.
Correspondingly, we illustrate that the counting result is invalid for graphs with given size
in terms of the signless Laplacian radius.
Indeed, it is not difficult to show that
if $G$  is a graph with $m\ge 4$ edges,
then $q(G) \le m+1$, where the equality holds if and only if $G$ is a star $K_{1,m}$.
Consequently, for a triangle-free graph $G$,
it follows trivially that $q(G) \le m+1$,
with equality  if and only if $G=K_{1,m}$.
These results can also be deduced
from a theorem of Zhai, Xue and Lou \cite{ZXL2020LAA}.
With the help of these results,
we obtain that if $G$ satisfies $q(G) \ge m+1$,
then $G=K_{1,m}$, which has no triangle. Moreover,  there  does not exist a graph $G$ with $m$ edges
satisfying $q(G) > m+1$.

\medskip

Let us mention a couple of problems related to our result.
Theorem \ref{thm-LLP}
has confirmed Conjecture \ref{conj-main}
 in the base case $q=1$ for triangles, and it can also be viewed as
 a spectral version of Theorem \ref{thmrad}.
In addition, one can also consider the case $q=1$ in Conjecture \ref{conj-main}
 by counting odd cycles.

\begin{problem}[Counting five-cycles]
If $G$ is an $n$-vertex graph with
\[  \lambda (G)
\ge \lambda (K_{\lceil \frac{n}{2} \rceil, \lfloor \frac{n}{2} \rfloor}^+), \]
then $G$ contains at least $\lfloor \frac{n}{2}\rfloor (\lfloor \frac{n}{2}\rfloor  -1)
(\lceil \frac{n}{2}\rceil  -2)$ copies of $C_5$.
\end{problem}

 Next, we intend to propose a possible generalization of Theorem \ref{thm-LLP}.
In 1962, Erd\H{o}s \cite{Erd1962a,Erd1962b}
 showed that there exists a small constant
 $c>0$ such that if $n$ is large enough and $1\le q<cn$, then
 every $n$-vertex graph with $\lfloor n^2/4\rfloor +q$
 edges   has at least $q\lfloor \frac{n}{2}\rfloor$ triangles.
 Moreover, Erd\H{o}s conjectured $c=\frac{1}{2}$, which was  confirmed by
Lov\'{a}sz and Simonovits \cite{LS1975,LS1983}.

\begin{theorem}[Lov\'{a}sz--Simonovits \cite{LS1975,LS1983}, 1975] \label{thm-LS1975}
Let $1\le q <\frac{n}{2}$ be a positive integer and $G$ be an $n$-vertex graph with
\[ e(G)\ge \left\lfloor \frac{n^2}{4}  \right\rfloor + q.\]
Then $G$ contains at least $q \lfloor \frac{n}{2}\rfloor $
triangles.
\end{theorem}

Incidentally, we point out that the condition $q < \frac{n}{2}$ is necessary,
since for even $n$ and $q = \frac{n}{2}$, one can add $ q +1$ extra edges
to the larger vertex part of the complete bipartite graph $K_{\frac{n}{2}+1, \frac{n}{2}-1}$ such that there is no triangle among these $q+1$ edges. Hence, we  get
a  graph with $\frac{n^2}{4} +q$ edges and
$(q+1)(\frac{n}{2} -1)=\frac{n^2}{4}-1$ triangles,
which is less than $q\lfloor \frac{n}{2}\rfloor $.
It is natural to propose a spectral version of Theorem \ref{thm-LS1975}.
Recall that $T_{n,2,q}$
is the graph obtained from
$T_{n,2}$
by embedding a star $K_{1,q}$ into the larger vertex part.
This suggests the following question.

\begin{problem}
Let $1\le q < \frac{n}{2}$ and $G$ be an $n$-vertex graph with
\[  \lambda (G)\ge \lambda (T_{n,2,q}) . \]
Then $ t(G) \ge q \lfloor \frac{n}{2}\rfloor $,
with equality  if and only if $G=T_{n,2,q}$.
\end{problem}

Finally, we conclude the spectral supersaturation result for all general graphs.
In 2009, Nikiforov \cite{Niki2009cpc} proved a spectral Erd\H{o}s--Stone--Bollob\'{a}s theorem, which implies that
for every graph $F$ with $\chi (F)=r+1$ and $\varepsilon >0$, there exists $n_0$ such that if $G$ is a graph on $n\ge n_0$ vertices
and $\lambda (G) \ge (1-\frac{1}{r} + \varepsilon )n$, then $G$ contains a copy of $F$.
More generally, we can extend this result as follows. 

\medskip 
\noindent 
{\bf Fact.}  
{\it For any $\varepsilon >0$, there exist $\delta>0$ and $n_0$ 
such that if $G$ has $n\ge n_0$ vertices and 
\[  \lambda (G)\ge \left(1- \frac{1}{r} + \varepsilon \right)n, \] 
 then 
 $G$ contains at least $\delta n^{f}$ copies of $F$, 
 where $f$ is the number of vertices of $F$. } 
 
 \medskip 
This spectral supersaturation can be guaranteed by the celebrated graph removal lemma, which states that
for every graph $F$ and  $\varepsilon >0$,
there exists $\delta =\delta (F, \varepsilon) >0$ such that
any graph on $n$ vertices with at most $\delta n^f$
copies of $F$ can be made $F$-free
by removing at most $\varepsilon n^2$ edges; see the comprehensive survey \cite{CF2013}.
Suppose that $G$ satisfies $\lambda (G) \ge (1-\frac{1}{r} + \varepsilon )n$
and $G$ has less than $\delta (F, {\varepsilon^2}/{8})\cdot n^f$ copies of $F$.
Then the graph removal lemma yields that
we can remove at most ${\varepsilon^2}n^2/8$ edges from $G$
such that the remaining subgraph $G_1$ is $F$-free.
Another famous result due to Erd\H{o}s, Frankl and R\"{o}dl \cite[Theorem 1.5]{EFR1986} states that if $n\ge n_0(F, {\varepsilon^2}/{8})$ and $G_1$ is an $n$-vertex $F$-free graph,  then we can remove at most ${\varepsilon^2} n^2/8$ edges from $G_1$
so that the resulting graph $G_2$ is $K_{r+1}$-free.
The Wilf theorem (see \cite{Wil1986} or \cite{Niki2002cpc}) gives that
$\lambda (G_2) \le (1-\frac{1}{r})n$. Consequently,
the Rayleigh formula gives
\begin{align*}
  \lambda (G) & \le
\lambda (G_2) + \lambda (G_1 \setminus G_2) + \lambda (G \setminus G_1)  \\
&<  \lambda (G_2) +  \sqrt{2e(G_1 \setminus G_2)}  +
\sqrt{2e(G \setminus G_1)}  \\
& \le \left( 1- \frac{1}{r} \right) n + \varepsilon n,
\end{align*}
a contradiction. Hence $G$ contains at least
$\delta(F, {\varepsilon^2}/{8})\cdot n^{f}$ copies
of $F$, as desired.

\subsection*{Acknowledgements}
This work was supported by the NSFC grant (Nos. 11931002 and 12371362),
the Natural Science Foundation of Hunan
Province (No. 2021JJ40707) and the Postdoctoral Fellowship Program of CPSF (No. GZC20233196).


\frenchspacing
\begin{thebibliography}{99}


\bibitem{Bollobas78}
B. Bollob\'as, Extremal Graph Theory, Academic Press, New York, 1978.



\bibitem{BN2007jctb}
B. Bollob\'{a}s, V. Nikiforov,
Cliques and the spectral radius,
J. Combin. Theory Ser. B 97 (2007) 859--865.


\bibitem{BM2008}
J.A. Bondy,  U.S.R. Murty,
Graph Theory, Vol. 244 of Graduate Texts in Mathematics, Springer, 2008.

\bibitem{C2006}
S. Cioab\u{a}, 
On the extreme eigenvalues of regular graphs, 
J. Combin. Theory Ser. B 96 (2006) 367--373. 

 \bibitem{CFTZ20}
 S. Cioab\u{a}, L. Feng, M. Tait, X.-D. Zhang,
 The maximum spectral radius of graphs without friendship subgraphs,
 Electron. J. Combin. 27 (4) (2020), No. P4.22.


  \bibitem{CDT2022a}
 S. Cioab\u{a}, D.N. Desai, M. Tait,
The spectral radius of graphs with no odd wheels,
European J. Combin. 99 (2022), No. 103420.


 \bibitem{CDT2022b}
 S. Cioab\u{a}, D.N. Desai, M. Tait,
 The spectral even cycle problem, 10 pages, (2022),
 arXiv:2205.00990. See \url{https://arxiv.org/abs/2205.00990}

 \bibitem{CDT2023}
  S. Cioab\u{a}, D.N. Desai, M. Tait,
 A spectral Erd\H{o}s--S\'{o}s theorem,
 SIAM J. Discrete Math. 37 (3) (2023) 2228--2239.


 \bibitem{CF2013}
D. Conlon, J. Fox, Graph removal lemmas,
in Surveys in Combinatorics,
London Math. Soc. Lecture Note Ser., vol. 409,
Cambridge Univ. Press, Cambridge, 2013, pp. 1--49.


\bibitem{Erd1955}
P. Erd\H{o}s, Some theorems on graphs,
Riveon Lematematika 9 (1955) 13--17.

\bibitem{Erd1962a}
P. Erd\H{o}s,
On a theorem of Rademacher--Tur\'{a}n,
Illinois J. Math. 6 (1962) 122--127.



\bibitem{Erd1962b}
P. Erd\H{o}s,
On the number of complete subgraphs contained in certain graphs, Magy. Tud. Acad. Mat. Kut. Int. Közl. 7
(1962) 459--474.

\bibitem{Erdos1964}
P. Erd\H{o}s,
On the number of triangles contained in certain graphs,
Canad. Math. Bull. 7 (1) (1964) 53--56.


\bibitem{EFR1986}
P. Erd\H{o}s,  P. Frankl, V. R\"{o}dl,
The asymptotic number of graphs not containing a fixed subgraph
and a problem for hypergraphs having no exponent,
Graphs Combin. 2 (2) (1986) 113--121.

\bibitem{FS13}
Z. F\"uredi,  M. Simonovits,
The history of degenerate (bipartite) extremal graph problems,
in Erd\H{o}s Centennial,
Bolyai Soc. Math. Stud., 25,
J\'{a}nos Bolyai Math. Soc., Budapest, 2013,
pp. 169--264.


\bibitem{HJZ2013}
B. He, Y.-L. Jin, X.-D. Zhang,
Sharp bounds for the signless Laplacian spectral radius in terms of clique  number,
Linear Algebra Appl. 438 (2013) 3851--3861.






\bibitem{LSY2022}
S. Li, W. Sun, Y. Yu,
Adjacency eigenvalues of graphs without short odd cycles,
Discrete Math. 345 (2022), No. 112633.


 \bibitem{LP2022second}
 Y. Li, Y. Peng,
 Refinement on  spectral Tur\'{a}n's theorem,
 SIAM J. Discrete Math. 37 (4) (2023) 2462--2485.

\bibitem{2022LLP}
Y. Li, L. Lu, Y. Peng,
Spectral extremal graphs for the bowtie,
  Discrete Math. 346 (12) (2023), No. 113680.


 \bibitem{LFP2023}
Y. Li, L. Feng, Y. Peng, 
A spectral extremal problem on non-bipartite triangle-free graphs, 
Electron. J. Combin. 31 (1) (2024), No. P1.52. 

\bibitem{LN2021outplanar}
H. Lin,  B. Ning,
A complete solution to the Cvetkovi\'{c}--Rowlinson conjecture,  
J. Graph Theory 97 (3) (2021) 441--450.

\bibitem{LNW2021}
H. Lin, B. Ning, B. Wu,
Eigenvalues and triangles in graphs,
Combin. Probab. Comput. 30 (2) (2021) 258--270.

\bibitem{LG2021}
H. Lin, H. Guo,
A spectral condition for odd cycles in non-bipartite graphs,
Linear Algebra Appl.  631 (2021) 83--93.

\bibitem{LS1975}
L. Lov\'{a}sz, M. Simonovits,
On the number of complete subgraphs of a graph,
in: Proc. of Fifth British Comb. Conf., Aberdeen, 1975,
pp.431--442.

\bibitem{LS1983}
L. Lov\'{a}sz, M. Simonovits,
On the Number of Complete Subgraphs of a Graph II,
in: Studies in Pure Math, Birkh\"{a}user (dedicated to P. Tur\'{a}n), 1983,
pp. 459--495.


 \bibitem{Lu2012}
 H. Lu,  
 Regular graphs, eigenvalues and regular factors, 
 J. Graph Theory 69 (2012) 349--355.


\bibitem{Man1907}
W. Mantel, Problem 28, Solution by H. Gouwentak, W. Mantel, J. Teixeira de Mattes, F. Schuh and W. A. Wythoff. Wiskundige Opgaven, 10 (1907) 60--61.


\bibitem{Mubayi2010}
D. Mubayi,
Counting substructures I: color critical graphs,
Adv. Math. 225 (2010) 2731--2740.


\bibitem{Niki2002cpc}
V. Nikiforov,
Some inequalities for the largest eigenvalue of a graph,
Combin. Probab. Comput. 11 (2002) 179--189.

\bibitem{Niki2006}
V. Nikiforov, Walks and the spectral radius of graphs, Linear Algebra Appl. 418 (2006) 257–268.

\bibitem{Niki2007laa2}
V. Nikiforov, Bounds on graph eigenvalues II,
Linear Algebra Appl. 427 (2007) 183--189.

\bibitem{Niki2009ejc}
V. Nikiforov,
Spectral saturation: inverting the spectral Tur\'{a}n theorem,
Electron. J. Combin. 16 (1) (2009), No. R33.

\bibitem{Niki2009cpc}
V. Nikiforov,
A spectral Erd\H{o}s--Stone--Bollob\'{a}s theorem,
Combin. Probab.  Comput. 18 (2009) 455--458.


\bibitem{NikifSurvey}
V. Nikiforov,  Some new results in extremal graph theory,
 Surveys in Combinatorics,  London Math. Soc. Lecture Note Ser., 392, Cambridge Univ. Press, Cambridge, 2011, pp. 141--181.



 \bibitem{NZ2021}
B. Ning, M. Zhai, Counting substructures and eigenvalues I: triangles,
European J. Combin.  110 (2023), No. 103685.

\bibitem{NZ2021b}
B. Ning, M. Zhai, Counting substructures and eigenvalues II: quadrilaterals,
14 pages, (2021), arXiv:2112.15279.
See \url{https://arxiv.org/abs/2112.15279}.



\bibitem{PY2017}
O. Pikhurko, Z. B. Yilma,
Supersaturation problem for color-critical graphs,
J. Combin. Theory Ser. B 123 (2017) 148--185.

 \bibitem{Sim1966}
M. Simonovits,
A method for solving extremal problems in graph theory,
stability problems, in: Theory of Graphs,
Proc. Colloq., Tihany, 1966, Academic Press, New York, (1968), pp. 279--319.


\bibitem{Sim13}
M. Simonovits,
Paul Erd\H{o}s' influence on extremal graph theory,
in The Mathematics of Paul Erd\H{o}s II,
R.L. Graham, Springer, New York, 2013, pp. 245--311.



\bibitem{TT2017}
M. Tait, J. Tobin,
Three conjectures in extremal spectral graph theory,
J. Combin. Theory Ser. B 126 (2017) 137--161.


\bibitem{Tait2019}
M. Tait, The Colin de Verdi\`{e}re parameter, excluded minors, and
the spectral radius, J. Combin. Theory Ser. A 166 (2019) 42--58.



\bibitem{WKX2023}
J. Wang, L. Kang, Y. Xue,
On a conjecture of spectral extremal problems,
J. Combin. Theory Ser. B 159 (2023) 20--41.


\bibitem{Wil1986}
H. Wilf, Spectral bounds for the clique and indendence numbers of graphs,
J. Combin. Theory Ser. B 40 (1986) 113--117.

\bibitem{WXH2005}
B. Wu, E. Xiao, Y. Hong,
The spectral radius of trees on $k$-pendant vertices,
Linear Algebra Appl. 395 (2005) 343--349.

\bibitem{XLGS2023}
J. Xue, R. Liu, J. Guo, J. Shu, 
The maximum spectral radius of irregular bipartite graphs, 
Adv. in Appl. Math.  142 (2023), No. 102433.   

\bibitem{ZXL2020LAA}
M. Zhai, J. Xue, Z. Lou,
The signless Laplacian spectral radius of graphs with a prescribed number of edges,
Linear Algebra Appl. 603 (2020) 154--165.

\bibitem{ZLS2021}
M. Zhai,  H. Lin,  J. Shu,
Spectral extrema of graphs with fixed size:
Cycles and complete bipartite graphs,
European J. Combin. 95 (2021), No. 103322.

\bibitem{ZLX2022}
M. Zhai, R. Liu, J. Xue,
A unique characterization of spectral extrema for friendship graphs,
Electron. J. Combin. 29 (3) (2022), No. 3.32.

\bibitem{ZL2022jgt}
M. Zhai, H. Lin,
A strengthening of the spectral color critical edge
theorem: books and theta graphs,
J. Graph Theory 102 (3) (2023) 502--520.

\bibitem{ZL2022jctb}
M. Zhai, H. Lin,  Spectral extrema of $K_{s,t}$-minor free graphs ----- on a conjecture of M. Tait,
J. Combin. Theory Ser. B 157 (2022) 184--215.

\end{thebibliography}
\end{document}